\documentclass[oneside,11pt]{article}
\usepackage{epsfig, amsfonts, amsmath, amsthm,amsbsy,subfigure}

\newtheorem{thm}{Theorem}
\newtheorem{example}{Example}

\usepackage{color}

\def\ve{\varepsilon}
\def\vr{\varepsilon}

 \DeclareMathOperator\erf{erf}

\begin{document}
\title{Singularly perturbed reaction-diffusion problems with discontinuities in the initial and/or the boundary data}

\author{J.L. Gracia\footnote{IUMA - Department of Applied Mathematics, University of Zaragoza, Spain. email: jlgracia@unizar.es.
The research of this author was partly supported by the Institute of Mathematics and Applications (IUMA), the
project MTM2016-75139-R and the Diputaci\'on General de Arag\'on (E24-17R).} \, and \,
E. O'Riordan\footnote{School of Mathematical Sciences, Dublin City
University, Ireland.
email: eugene.oriordan@dcu.ie}}

\date{\today}



\maketitle

{\bf Abstract:}
Numerical approximations to the solutions of three different problem classes of  singularly perturbed parabolic reaction-diffusion problems, each with a discontinuity in  the bound\-ary-initial data,  are generated. For each problem class, an analytical function  associated with the discontinuity in the data, is identified. Parameter-uniform  numerical approximations to the difference between the analytical function and the solution of the singularly perturbed problem  are generated using  piecewise-uniform Shishkin meshes.  Numerical results are given to illustrate all the theoretical error bounds established in the paper.

\section{Introduction}

 To establish theoretical error bounds for a numerical method,  one  requires the solution of the continuous problem to
be sufficiently regular, in order  that certain partial derivatives of the solution are bounded on the closed domain. For parabolic problems, this often requires the assumption of sufficient compatibility conditions between the initial and boundary data, which can be viewed as solely a theoretical \cite{flyer1} constraint. However, accuracy can be lost in the numerical approximations if insufficient compatibility is imposed \cite{temam,flyer2}, especially if  a higher order numerical method is utilized. Moreover, certain mathematical models  (e.g., Biot's consolidation theory of porous media~\cite{biot1,biot2}) consider mathematical models with discontinuities between the initial and boundary data deliberately built into the problem formulation. In this paper, we consider the effect of discontinuous boundary/initial data on numerical approximations, in the context of singularly perturbed parabolic problems.

The solutions of singularly perturbed problems, with smooth data, typically exhibit boundary layers, whose widths depend on the singular perturbation parameter.
Additional interior layers can appear when the coefficients of the differential operator are discontinuous or if the inhomogenous term contains a point source \cite{bulg-interior,disc-rd1,perth,shishkin3}.
In all of these problem classes, parameter-uniform numerical methods \cite{fhmos} have been  constructed, by using {\it a priori} information about both the location and character of all boundary/interior layers that are present in the solution and using this analytical information to design an appropriate piecewise-uniform Shishkin mesh \cite{fhmos}  for the problem.

 In the main, the data for the problem  need to be sufficiently smooth, in order to prevent further classical singularities  appearing in the solution. As the smoothness of the data reduces, then the order of convergence can also reduce \cite{shish-2007weak,shish-2007weak-Steklov,Zhemukhov1,Zhemukhov2}. In the case of sufficiently smooth data satisfying  second-order compatibility conditions at the corners of the space-time domain (see Appendix 1), the typical error bound \cite{ria}  in the $L_\infty$ norm for singularly perturbed parabolic problems of reaction-diffusion type,  is of the form
\[
\Vert \bar U -u \Vert \leq C (N^{-1} \ln N )^2 + C M^{-1},
\]
when one uses a tensor product of an appropriate piecewise-uniform Shishkin mesh in space (with $N$ elements) and a uniform mesh (with $M$ elements) in time, to generate a global approximation $\bar U$ to the continuous solution $u$. If there is only  zero-order compatibility conditions assumed and a possible jump in the first time derivative of the boundary data, then the same numerical method retains parameter-uniform convergence, albeit with some minor reduction in the order \cite{shish-2007weak} of convergence in space. If there is a discontinuity in the first space derivative of the initial condition, the order of convergence can drop to $O(N^{-1} + M^{-0.5})$ \cite{shish-2007weak-Steklov}, \cite[\S 14.2]{Shish-redbook}. Nevertheless, the numerical method (based on an appropriate piecewise-uniform mesh)  retains parameter-uniform convergence of some positive order.

 However, for the singularly perturbed heat equation
\[
-\ve u _{xx} +u_t = 0, \ (x,t) \in (0,1)\times (0,T],
\]
if the initial condition $u(x,0)$ is discontinuous \cite{bulg2012-1} or if there is an incompatibility between the initial and boundary conditions $u(0^+,0) \neq u(0,0^+)$
then rectangular meshes do not produce parameter-uniform numerical methods \cite{hemker,hemker2}. Note that if one incorporates a co-ordinate system aligned to the similarity transformation $\theta = x/(2\sqrt{\ve t})$, then one can design a piecewise-uniform mesh \cite{shishkin2} in this transformed co-ordinate system, to generate parameter-uniform numerical approximations. However, we will not consider the use of such  transformed co-ordinate systems here.

In this paper, we introduce a mixed analytical/numerical method, which is based on the ideas in \cite{flyer2}. This method
first  identifies explicitly  the main singular component $s(x,t)$ associated with the singularity and uses a piecewise-uniform Shishkin mesh to generate a parameter-uniform numerical approximation to the difference, $u-s$, between the exact
 solution $u$ and the main singular component $s$. In this way, parameter-uniform numerical approximations are created for singularly perturbed problems with discontinuous initial conditions, problems with incompatible initial/boundary conditions and problems with discontinuous boundary conditions.

Below we  examine singularly perturbed problems of the form
\[
u_t-\ve u_{xx} +b(t) u =f(x,t),\quad (x,t) \in Q:=(0,1) \times (0,T]; \quad   b(t) \ge 0, \ t \geq 0,
\]
where the boundary/initial data will have a discontinuity at some point on the boundary $\bar Q \setminus Q$. Note that the coefficient $b(t)$ is assumed to be independent of the space variable $x$. This assumption permits us present relatively simple proofs for all of the pointwise bounds on the derivatives of the components of the continuous solutions, presented below. In \cite{go_dbb}, a related problem class was considered, which involved the differential equation
\[
\ve (u_t-u_{xx}) +b(x,t) u =f, (x,t) \in (0,1) \times (0,1]; \quad  b(x,t) >0;
\]
with  incompatible boundary-initial data.
Note that the coefficient $b(x,t)$ can vary in space. However, the corresponding proofs (establishing bounds on the derivatives of the layer components) are significantly longer and contain much more technical detail to what is required for the three problem classes considered in the current paper. Nevertheless, in our numerical results section, we present test examples where the coefficient $b(x,t)$ does vary in space and we see that (from a computational perspective) the assumption $b(t)$ appears not necessary, in practice. In summary, the assumption $b(t)$ allows us present theoretical error bounds for three problem classes in a single publication. In this way, we see the minor modifications in the overall approach, when dealing with singularly perturbed problems with discontinuous data. In addition, initial layers appeared in the problem class studied in \cite{go_dbb}, which required the use of a Shishkin mesh in time. In the current paper, a uniform mesh in time suffices, as the time derivative of the continuous solution has a coefficient of order one in this paper.

 The paper is structured as follows: In \S \ref{sec:continuous} the asymptotic behaviour of the solution $u$ of the three classes problems is analysed. For each class of problems the singular component $s$ is identified and the behaviour of $u-s$ is revealed by using an appropriate decomposition into regular and layer components. In \S \ref{sec:numericalmethod}  a finite difference scheme is proposed to approximate $u-s$ for each class of problems. Each scheme uses the backward Euler method in time and standard central differences in space defined on appropriately constructed meshes of Shishkin type. Error estimates in the maximum norm are established,  which yield  global parameter-uniform convergence for the methods. In \S \ref{sec:numerical} some numerical results for the three classes of problems are given and they indicate that our error estimates are sharp. The paper is completed with two technical appendices.

\noindent {\bf Notation.}
Throughout the paper,   $C$  denotes a generic constant that is independent of the singular perturbation parameter $\ve$ and of all discretization parameters.
  The $L_\infty$ norm on the domain $D$  will  be denoted by $\Vert \cdot \Vert _D$  and the subscript is omitted if the domain is $\bar Q$.

\section{Three classes of problem} \label{sec:continuous}

Before we define the three problem classes to be examined in this paper, we define  a set of singular functions which are associated with the singularities that are generated by discontinuous boundary/initial data in singularly perturbed problems.

The  singular function $s:(-\infty,\infty)\times [0,\infty) \rightarrow (-1,1)$ is defined as
 \begin{equation}\label{definition-of-s}
s(x,t):=e^{-b(0) t}\erf\left(\frac{x}{2\sqrt{\vr t}} \right), \ \hbox{where} \ \erf(z):= \frac{2}{\sqrt{\pi}} \int _{r=0}^z e^{-r^2} \ dr.\end{equation}
 This function satisfies  the constant coefficient quarter plane homogeneous problem
\begin{align*}
& s_t-\vr s_{xx} +b(0)s =0, \quad (x,t) \in (0,\infty) \times (0,\infty), \\
&s(0,t)=0,\quad t > 0; \quad
s(x,0)=1, \quad x >0.
\end{align*}
Observe that $s \not \in C^0([0,\infty) \times [0,\infty))$. Define the associated set of  functions
\[
s_n(x,t) := t^ns(x,t), \quad n \geq 0;\quad cs_n(x,t):= t^n(1-s(x,t)).
\] Then $s_n,cs_n \in C^{n-1+\gamma}([0,\infty) \times [0,\infty)),\ n \geq 1$. \footnote{The space $C^{n+\gamma}(\bar Q ) $ is
the set of all functions, whose derivatives of
order $n$ are H\"{o}lder continuous of degree $\gamma >0$. That is,
\[
C^{n+\gamma } ( \bar Q ) := \left \{ z : \frac{\partial ^{i+j} z}{
\partial x^i
\partial t^j } \in C^{\gamma }(\bar Q), \ 0 \leq i+2j \leq n \right \} .
\]} Note further that
\begin{align*}
&(s_n) _t-\vr (s_n)_{xx} +b(0)(s_n) =n(s_{n-1}) \\ \hbox{and} \quad
& (cs_n) _t-\vr (cs_n)_{xx} +b(0)(cs_n) =n(cs_{n-1}) +b(0)t^n.
\end{align*}
Hence, for all $n \geq 1$, we have the recurrence relationship
\begin{equation}\label{useful}
Ls_n(x,t) = n s_{n-1}(x,t) +(b(t)-b(0))s_n(x,t),
\end{equation}
 where
$
Lz:=z_t- \vr z_{xx}+b(t)z.
$

\subsection{Problem Class 1: incompatible boundary-initial data}
Consider  the singularly perturbed parabolic problem: Find $u:\bar Q \rightarrow \mathbb{R}$ with $Q:=(0,1)\times(0,T],$ such that
\begin{subequations} \label{Cproblem}
\begin{align}
& Lu =  u_t- \vr u_{xx}+b(t)u=f(x,t), \ \text{in } Q;\quad  b(t)\ge \beta \geq 0,  \forall t \geq 0;\\
&u(0,t)=0 , \ u(1,t)=0, \ t\ge 0, \quad u(x,0)=\phi(x), \ 0<x<1; \\
&\phi(0^+)\ne 0, \ \phi(1)=0, \quad  f,b \in C^{4+\gamma }(\bar Q ),  \ \ \phi \in C^4(0,1);
 \\
& f(1,0)=- \ve \phi ''(1^-),\quad  -\ve \phi ^{(iv)}(1^-) +b(0)\phi ''(1^-) = (f_t+f_{xx})(1,0); \label{comp} \\
& f(0,0)= -\ve \phi ''(0^+).  \label{extra-1}
\end{align}
\end{subequations}
Since this problem is linear, there is no loss in generality in assuming homogeneous boundary conditions.  Observe that there is a discontinuity in the data at the corner point $(0,0)$.
The discontinuity in the data for this first problem class is the same discontinuity as that examined in  \cite{go_dbb}.  The assumption ~\eqref{extra-1} on the data allows us present a simplified version of the numerical analysis. Without this assumption, we would require more of the analysis from \cite{go_dbb}.   By assuming the compatibility conditions (\ref{comp}), we prevent any classical singularities appearing in the vicinity of the point $(1,0)$ (see Appendix 1).

 In order to deduce the asymptotic behaviour of the solution of problem~\eqref{Cproblem}, it is decomposed into the sum
\begin{equation}\label{decomp1}
u=\phi(0^+) s(x,t) +y.
\end{equation}
Note that $\vert \phi(0^+) \vert$ is the magnitude of the jump in the boundary/initial data, at $(0,0)$. The remainder  $y$, defined by (\ref{decomp1}), satisfies the problem
\begin{subequations}\label{eq:ComponentY2}
\begin{eqnarray}
&Ly=F:=f-(b(t) - b(0))\phi(0^+)s, \ (x,t)\in Q; \\
&y(0,t)=0, \ y(1,t)=-\phi(0^+) s(1,t), \ t\ge 0; \\ & y(x,0)=\phi(x)-\phi(0^+), \ 0<x<1.
\end{eqnarray}
\end{subequations}
Recall that $\phi (1)=0$ and so $y(1^-,0)=y(1,0^+)$. Hence the boundary and initial data are continuous in the case of problem (\ref{eq:ComponentY2}),
$F  \in C^{0+\gamma} (\bar Q)$ and,   using assumption ~\eqref{extra-1}, we have that  $y  \in C^{2+\gamma} (\bar Q)$.
We further  decompose the solution of (\ref{eq:ComponentY2})  as follows:
\begin{subequations}\label{decomp}
\begin{equation}
y= v+w_L+w_R,
\end{equation}
where the regular component $v$ satisfies the problem
\begin{equation}\label{regular}
L^*v^*=f^*-(b(t)-b(0))\phi (0^+), \ (x,t)\in Q_0^*:=(-a,1+a) \times (0,T],
 \end{equation}
which is posed on an extended (in the spatial direction) domain $Q_0^*$ {\footnote{We use the notation $f^*:\bar Q^* \rightarrow \mathbb{R}$ to denote the extension of any function $f:\bar Q \rightarrow \mathbb{R}$ such that $f^* (x,t) \equiv f(x,t), \ (x,t) \in \bar Q$  and $\bar Q \subset \bar Q^*$.}}  and $a$ is an arbitrary positive parameter. The initial/boundary values for the regular component are determined by  $v^* =v^*_0+\vr v^*_1$, where these two subcomponents, in turn,  satisfy
\begin{eqnarray}
&bv^*_0+(v^*_0)_t=f^*-(b(t)-b(0))\phi (0^+),\quad t>0;\\
& v^*_0(x,0)=y^*(x,0), \quad  x \in (-a,1+a) , \\
&L^*v_1^* = (v^*_0)_{xx},\quad (x,t)\in Q_0^*, \quad v^*_1=0, \ (x,t)\in \partial  Q_0^*. \label{eq5e}
\end{eqnarray}
Observe that the singular function $s$ is not involved in the definition of the regular component. Moreover, observe that
\[
v_0(0,t) = \int _{s=0}^t \bigl( f(0,s) - (b(s)-b(0) \phi (0^+) \bigr) e^{-(b(t)-b(s))} \ ds ,\qquad t \geq 0.
\]
The boundary layer components $w_L,w_R$ satisfy the problems
\begin{eqnarray}
&L^* w^*_L=F^*-L^* v^* , \ (x,t)\in Q_1^*:=(0,1+a) \times (0,T],\\ &  w^*_L(0,t)=-v^*(0,t), \  w^*_L(x,0)=0, \ w^*_L(1+a,t)=0;\\
&Lw_R=0, \ (x,t)\in Q,\\ & w_R(1,t)=-(v^*+w^*_L)(1,t), \  w_R(x,0)=0, \ w_R(0,t)=0.
\end{eqnarray}
\end{subequations}
By construction  and by using the extended domains to avoid compatibility issues, $v, w_R \in C^{4+\gamma}(\bar Q)$. Note that $ L^*w^*_L(0,0) =0$ and $v_t(0,0)=f(0,0)$, so the first order compatibility conditions are satisfied (see Appendix 1). Hence,  $w_L \in C^{2+\gamma}(\bar Q)$; but, in general, $w_L \not\in C^{4+\gamma}(\bar Q)$.
\begin{thm} (Problem Class 1)
\begin{subequations}
For the regular component $v$ we have, for all $0 \leq i+2j \leq 4$ with $0 < \mu < 1$, the bounds
\begin{eqnarray}\label{vbounds}
\left\Vert \frac{\partial ^{i+j}v}{\partial x^{i}\partial t ^{j}}  \right\Vert  \leq C (1+\vr ^{1-(i/2 )}),
\end{eqnarray}
and for the boundary layer components, we have the bounds
\begin{align}
 \left\vert \frac{\partial ^{i+j}}{\partial x^{i}\partial t ^{j}} w_{ R}(x,t) \right\vert &\leq C \vr ^{-(i/2 )}e^{-\frac{(1-x)}{\sqrt{\vr}}}, \quad  0 \leq i+2j \leq 4;
 \quad (x,t) \in \bar Q;\label{expRlayer-bounds} \\
 \left\vert \frac{\partial ^{i+j}}{\partial x^{i}\partial t ^{j}} w_{L}(x,t)\right \vert  &\leq C \vr ^{-(i/2)}e^{-\frac{\mu}{2} \frac{x}{\sqrt{T\vr}}}, \quad   0 \leq i+2j \leq 2;  \quad (x,t) \in \bar Q.\label{expLlayer-xbounds}
\end{align}
In addition, for the higher derivatives, we have
\begin{eqnarray}
\left \vert \frac{\partial ^{i}}{\partial x^{i}} w_{L}(x,t) \right \vert &\leq&  \frac{C}{\vr (\sqrt{\vr t})^{i-2} } e^{-\frac{\mu}{2} \frac{x}{\sqrt{T\vr}}}, \quad i=3,4; , \quad (x,t) \in Q;\label{expLlayer-4th-xbounds}\\
\left \vert \frac{\partial^2}{\partial t^2} w_{L}(x,t)\right \vert &\leq& \frac{C}{ t}, \quad (x,t) \in Q. \label{expLlayer-2nd-tbounds}
\end{eqnarray}
\end{subequations}
\end{thm}

\begin{proof} We begin by establishing the bounds on the regular component.
 Note that $v_0$ is bounded independently of $\ve$. Consider the problem (\ref{eq5e}) transformed with the stretched variable
   $\frac{x}{\sqrt{\ve}}$.
Apply the {\it a priori} bounds \cite{ladyz} to establish bounds on the partial derivatives of $v_1$ in the stretched variables. Transforming back to the original variables, we deduce the bounds  (\ref{vbounds}). The bounds on $w_R$ are obtained in the usual way \cite{ria}.

 We now consider the component $w_L$. Observe that, with $cs_n:=t^n(1-s)$, we have
\[
L^* w^*_L =(b(t)-b(0) -tb'(0))\phi(0^+)cs_0+ b'(0)\phi(0^+)cs_1, \ (x,t) \in  Q_1^*;
\]
and $L^* w^*_L \in C^{0+\gamma}(\bar Q_1^*)$. Using $e^{-z^2} \leq e^{0.25-z}$, it follows that
\begin{align*}
&\vert cs_0(x,t) \vert \leq  Ce^{-\frac{x}{2\sqrt{ \ve T}}} , \qquad  \vert cs_1(x,t) \vert \leq  Cte^{-\frac{x}{2\sqrt{ \ve T}}},\ t \leq T;\\
&\left\vert \frac{\partial }{\partial t} cs_1(x,t)\right\vert \leq C e^{-\mu\sqrt{\frac{1 }{2t\vr}}x}, \qquad \mu <1.
\end{align*}
Hence, using a maximum principle, we can deduce that
\[
\vert w_L(x,t) \vert \leq C e^{-\frac{x}{2\sqrt{ \ve T}}}e^{\theta t}, \quad \theta > \frac{1}{4T} - \beta, \quad  (x,t) \in Q;
\]
and, by applying the arguments from \cite{ria}, we get that for $0\leq i+2j \leq 2$,
\[
 \left\vert \frac{\partial ^{i+j}}{\partial x^{i}\partial t ^{j}} w_{ L}(x,t) \right\vert \leq C \vr ^{-(i/2 )}e^{-\frac{\mu x }{2\sqrt{T\vr}}}, \quad (x,t) \in Q.
\]
To obtain bounds on the higher derivatives of $w_L$, we introduce a further decomposition  of this boundary  layer function.
Consider the continuous function
\begin{eqnarray*}
P(x,t) &:=& B(t)- \int _{r=0}^t s_0(x,r) \ dr ,  \\
 \hbox{where} \quad B(t)  &:=&\begin{cases} \frac{1-e^{-b(0)t}}{b(0)}, & \hbox{ if } \ b(0) \neq 0, \\
t, & \hbox{ if } \  b(0)=0 .
\end{cases}
\end{eqnarray*}
This function has been constructed to satisfy the following problem
\begin{align*}
& P_t-\vr P_{xx} +b(0)P=0, \quad \text{in } Q, \\
&P(0,t)=B(t)  \quad t\ge 0, \quad P(x,0)=0, \quad 0<x<1.
\end{align*}
Note that
\[
P(x,t)  = B(t)-t + \int _{r=0}^t cs_0(x,r) \ dr.
\]
We introduce the secondary expansion
\[
w^*_L(x,t) =-v^*_t(0,0) P^*(x,t) + b'(0) \phi (0^+) \frac{1}{2} cs_2(x,t) +R^*(x,t) ;
\]
and the remainder term $R^*$, defined over $\bar Q^*_1$, satisfies the problem
\begin{align*}
L^*R^* &=  (b(t)-b(0)-b'(0)t)\phi (0^+) cs_0  + (b(t)-b(0))  v^*_t(0,0) P^* \\
&  - (b(t)-b(0))b'(0)\phi (0^+)\frac{1}{2} cs_2,  \\
 R^*(x,0)&=0, \ 0 < x < 1+a; \quad  R^*(1+a,t) =R^*(1+a,t), \ t \geq 0; \\
 R^*(0,t) &= (t v^*_t(0,0)- v^*(0,t)) +v^*_t(0,0)  (B(t) -t)-b'(0)\phi (0^+)\frac{t^2}{2}.
\end{align*}
Note that $v(0,0)=0$ and $\vert R^*(0,t) \vert \leq C t^2$. Moreover, using properties of $s_2$ (see Appendix 2) one can check that  $L^*R^* \in C^{2+\gamma} (\bar Q^*_1)$ and second level compatibility is satisfied at the point $(0,0)$. Hence, we have $R^* \in C^{4+\gamma} (\bar Q^*_1)$. Using this regularity, we can deduce the bounds
\[
\left\Vert \frac{\partial ^{i+j}R^*}{\partial x^{i}\partial t ^{j}}  \right\Vert  \leq C (1+\vr ^{-(i/2 )}), \quad  0\le i+2j\le 4.
\]
From this, we obtain
\[
\left \vert \frac{\partial ^{2}}{\partial t ^{2}} w_{L}(x,t)\right\vert  \leq Ct^{-1}, \quad (x,t) \in Q.
\]
To obtain the desired bounds (involving the decaying exponential) on the third and fourth space derivatives of $w_L$, we form a problem for $R^*_2:=\ve R^*_{xx}$ by differentiating the differential equation satisfied by $R^*$  twice to get
\begin{align*}
L^*R_2^* &=  (b(0)+b'(0)t-b(t))\phi (0^+) \ve(cs_0 )_{xx} + (b(t)-b(0)) v_t(0,0) \ve P_{xx}^*  \\
&  - (b(t)-b(0))b'(0)\phi(0^+)\ve \frac1{2} (cs_2)_{xx},\\
 R_2^*(x,0) &=0, R_2^*(1+a,t) =R_2^*(1+a,t); \quad
R_2^*(0,t) =(g'+bg- h)(t);
\end{align*}
where $ g(t):=R^*(0,t),\ h(t) :=L^*R^*(0,t)$ are both smooth functions independent of $\ve$. We can complete the proof (as in \cite{ria}) by noting that
\begin{eqnarray*}
 \left \vert  L^*R^*_2  (x,t) \right\vert &\leq&  Ce^{-\mu \frac{x}{\sqrt{4\ve T}}};\\
 \left\vert  (L^*R^*_2) _{t} (x,t) \right\vert+  \left \vert \ve (L^*R^*_2)_{xx} (x,t) \right\vert &\leq& Ce^{-\mu \frac{x}{\sqrt{4\ve T}}}.
\end{eqnarray*}

\end{proof}

%

\subsection{Problem Class 2: discontinuous initial condition}

Consider  the singularly perturbed parabolic problem
\begin{subequations} \label{CproblemDiscontinuous}
\begin{align}
&Lu =f(x,t), \ (x,t) \ \text{in } Q;\quad
u(0,t)=0, \ u(1,t)=0, \ t\ge 0; \\
& u(x,0)=\phi(x), \ 0<x<1, \quad \phi (0) = \phi (1) =0; \label{comp0}\\
& \phi(d^-)\ne \phi(d^+), \quad 0 < d <1; \\
& f(0,0)=- \ve \phi ''(0^+), -\ve \phi ^{(iv)}(0^+) +b(0)\phi ''(0^+) = (f_t+f_{xx})(0,0); \label{comp1} \\
& f(1,0)=- \ve \phi ''(1^-), -\ve \phi ^{(iv)}(1^-) +b(0)\phi ''(1^-) = (f_t+f_{xx})(1,0); \label{comp2}\\
&f,b \in C^{4+\gamma }(\bar Q ),  \ \ \phi \in C^4((0,1) \setminus \{d \}) \\
&\phi '(d^-)=  \phi '(d^+), \quad \phi ''(d^-)=  \phi ''(d^+) \label{extra-2}.
\end{align}
\end{subequations}
The assumption of the compatibility conditions (\ref{comp0}), (\ref{comp1}) and (\ref{comp2}) ensures that  no classical singularity appears near the corner points
 $(0,0), (1,0)$. However, observe that the initial function $\phi(x)$ is discontinuous at $x=d$. This will cause an interior layer to appear in the solution, near the point $(d,0)$.
 The assumption $\phi '(d^-)=  \phi '(d^+)$ on the data prevents a drop in the order of convergence in our numerical approximations, as in the case of \cite[\S 14.2]{Shish-redbook}.

Decompose the solution of~\eqref{CproblemDiscontinuous} into the following sum
\begin{equation}\label{decomp2}
u(x,t)=\frac{[\phi](d)}{2}s(x-d,t)+y(x,t), \quad \text{where } \  [\phi](d):=\phi(d^+)-\phi(d^-).
\end{equation}
By the definition (\ref{definition-of-s}) of the discontinuous function $s$, we have that
$$
s(x-d,0)=
\begin{cases}
-1, & \hbox{ for } x <d, \\
0, & \hbox{ for } x=d, \\
1 , & \hbox{ for } x > d.
\end{cases}
$$
The component  $y$ is the solution of the problem
\begin{subequations}\label{eq:ComponentY2Discontinuous}
\begin{align}
&Ly=f-(b(t) - b(0))0.5[\phi](d) s(x-d,t), \ (x,t)\in Q, \\
&y(0,t)=-0.5[\phi](d) s(-d,t), \ y(1,t)=-0.5[\phi](d) s(1-d,t), \quad t\ge 0, \\
&y(x,0)=\phi(x)-0.5[\phi](d)s(x-d,0),\ x \neq d; \\
&  y(d,0) = (\phi(d^+)+\phi(d^-))/2;
\end{align}
\end{subequations}
and $y(x,0)$ is continuous for all $x \in [0,1]$. Moreover, due to (\ref{extra-2}), we have $y(x,0) \in C^2(0,1)$.
Using the maximum principle
\[
\Vert y \Vert \leq C.
\]
The solution $y$ is further decomposed into the sum
\[
y=v+w_L+w_R+w_I,
\]
 where the components $v$ and $w_I$ are discontinuous functions and the components $w_L$ and $w_R$ are continuous functions.
The regular component $v$ is constructed to satisfy the problem
\begin{equation}\label{regular2}
L^*v^*=f^*-(b(t)-b(0))0.5[\phi](d)e^{-b(0)t} s(x-d, 0), \ (x,t)\in Q^*.
 \end{equation}
This problem is posed on the extended domain $Q^*:=(-a,1+a) \times (0,T], a > 0$. The initial/boundary values for the regular component are determined by  $v^* =v^*_0+\vr v^*_1$, where the reduced solution $v_0$ satisfies the initial value problem
\begin{subequations}
\begin{align}
bv^*_0+(v^*_0)_t&=f^* -(b(t)-b(0))0.5[\phi](d) e^{-b(0)t}s(x-d,0),\ t>0; \\
v^*_0(x,0)&=y^*(x,0),\ x \in (-a,1+a).
\end{align}
\end{subequations}
Observe that the reduced solution $v^*_0$ is continuous, but in general \[ (v^*_0)_{x}(d^+,0) \neq (v^*_0)_{x}(d^-,0).\] The first correction $v_1$ is defined as the multi-valued function
$$
v^*_1(x,t):=
\begin{cases}
v_1^-, & \hbox{ for } x\leq d, \ t \geq 0, \\
v_1^+, & \hbox{ for } x\geq d,   \ t \geq 0,
\end{cases}
$$
and the two sides of this function are the solutions of
\begin{eqnarray*}
L^*v_1^- = (v^*_0)_{xx},\quad  -a < x\leq d,\ t >0 ;\\ v_1^-(-a,t) =0, \ t \geq 0; \quad v_1^-(x,0) =0, \ x \leq d; \\
\\L^*v_1^+ = (v^*_0)_{xx},\  \quad  d \leq  x < 1+a,\ t >0 ;\\ v_1^+(1+a,t) =0, \ t \geq 0; \quad v_1^+(x,0) =0, \ x \geq d;
\end{eqnarray*}
where we use $ Lv^-_1(d,t) := (v_0)_{xx}(d^-,t)$ and $ Lv^+_1(d,t) := (v_0)_{xx}(d^+,t)$. Since the regular component $v$ is multi-valued we now define
the subdomains
\[
Q^-:=(0,d) \times (0,T] \quad \hbox{and} \quad  Q^+:=(d,1) \times (0,T].
\]
By using suitable extensions to these  subdomains, we can have $v^\pm_1 \in C^{4+\gamma} (\bar Q^\pm)$. For example,
\begin{align*}
&L^*  (v_1^-)^* = ((v_0)_{xx})^*,\quad  (x,t) \in Q_*^-:=(-a, d+a) \times (0,T]; \\
& (v_1^-)^*(-a,t) =   (v_1^-)^* (d+a,t) =0, \ t\ge 0, \quad
 (v_1^-)^*(x,0) = 0, \quad -a<x<d+a.
\end{align*}


The boundary layer components $w_L,w_R$ satisfy the problems
\begin{subequations}\label{bdy-layers2}
\begin{eqnarray}
&Lw_L=Lw_R=0 , \ (x,t)\in Q ,\\ & w_L(0,t)=-v(0,t), \  w_L(x,0)=0, \ w_L(1,t)=0,\\
 &  w_R(0,t)=0,  \  w_R(x,0)=0, \ w_R(1,t)=-v(1,t).
\end{eqnarray}\end{subequations}

\begin{thm} (Problem Class 2)
\begin{subequations}
For the regular component $v$ we have, for all $0 \leq i+2j \leq 4$, the bounds
\begin{eqnarray}
\left \Vert \frac{\partial ^{i+j}v^- }{\partial x^{i}\partial t ^{j}} \right\Vert _{\bar Q^-} ,
\left \Vert \frac{\partial ^{i+j}v^+}{\partial x^{i}\partial t ^{j}}  \right \Vert _{\bar Q^+} \leq C (1+\vr ^{1-(i/2 )}).
\end{eqnarray}
For the boundary layer components,  for all $0 \leq i+2j \leq 4$   and $(x,t) \in \bar Q$,
\begin{eqnarray}
 \left\vert \frac{\partial ^{i+j}w_{L}}{\partial x^{i}\partial t ^{j}} (x,t)\right\vert  \leq C \vr ^{-(i/2)}e^{- \frac{x}{\sqrt{\vr}}};\
 \left\vert \frac{\partial ^{i+j}w_{ R}}{\partial x^{i}\partial t ^{j}} (x,t) \right\vert \leq C \vr ^{-(i/2 )}e^{-\frac{(1-x)}{\sqrt{\vr}}}.
\end{eqnarray}
\end{subequations}
\end{thm}
\begin{proof} Adapt appropriately the argument from the proof of Theorem 1.
\end{proof}

Finally the multi-valued interior layer component
$$
w_I(x,t):=
\begin{cases}
w_I^-, & \hbox{ for } 0\leq x \leq d, \ t \geq 0, \\
w_I^+, & \hbox{ for }  d \leq x \leq 1, \ t \geq 0;
\end{cases}
$$
is defined implicitly by the sum $y=v+w_L+w_R+w_I$.
Hence,
$
\Vert w_I \Vert \leq C.
$
Moreover,  $w_I$ satisfies the problem
\begin{subequations}\label{interior-layer}
\begin{equation}
Lw_I=R(x,t), \ (x,t)\in Q ^- \cup Q^+,
\end{equation}
where
\begin{align}
&R(x,t):= (b(t)-b(0))0.5[\phi](d)\left(e^{-b(0)t}s(x-d,0)-s(x-d,t)\right),\\
& w_I(0,t)=0, \   w_I(1,t)=0, t \geq 0; \quad w_I(x,0)=0, \ 0 \leq x \leq 1; \\
&[w_I](d,t) = -\ve [ v_1](d,t), \qquad [(w_I)_x  ] (d,t)= -\ve [(v_1)_x ](d,t).
\end{align}
\end{subequations}
Note that, $R \in C^{0+\gamma}(\bar Q)$, $R(d,t) \neq 0 $ for all $t >0$ such that $b(t) \neq b(d)$. Moreover,
\[
\vert R (x,t) \vert 
 \leq Ct e^{-\frac{(x-d)^2}{4\ve t}} \leq Ct e^{-\frac{\vert x-d \vert}{2\sqrt{\ve T}}}, \ (x,t) \in \bar Q
\]
and
\[
\left \vert \frac{\partial}{\partial t} R (x,t) \right\vert \leq C e^{-\frac{\mu \vert x-d \vert}{2\sqrt{\ve t}}}, \ \mu < 1,  \  (x,t) \in Q.
\]
Using a maximum principle, either side of $x=d$, we have for $0 \leq i+2j \leq 2$
\begin{subequations}
\begin{align}
 \left\vert  w_{I}(x,t)\right\vert  &\leq C e^{-\frac{ \vert x -d\vert}{2\sqrt{\ve T}}};\quad (x,t) \in \bar Q^- \cup \bar Q^+;\\
 \left\vert \frac{\partial ^{i+j}}{\partial x^{i}\partial t ^{j}} w_{ I}(x,t) \right\vert &\leq C \vr ^{-(i/2 )}e^{-\frac{\mu \vert x -d\vert}{2\sqrt{\ve T}}}, \ \mu < 1,
\quad  (x,t) \in \bar Q^- \cup \bar Q^+.
\end{align}
 For the higher derivatives,  we need to repeat the argument from the proof of Theorem 1, from the  last section, to establish the additional bounds
\begin{align}
\left \vert \frac{\partial ^{i}}{\partial x^{i}} w_{I}(x,t) \right \vert &\leq  \frac{C}{\vr (\sqrt{\vr t})^{i-2} } e^{-\frac{\mu}{2} \frac{\vert x -d\vert}{\sqrt{T\vr}}}, \quad i=3,4, \quad (x,t) \in  Q^- \cup  Q^+; \\
\left \vert \frac{\partial^2}{\partial t^2} w_{I}(x,t)\right \vert &\leq \frac{C}{ t}, \quad (x,t) \in  Q^- \cup  Q^+.
\end{align}
\end{subequations}

\subsection{Problem Class 3: discontinuous boundary data}

Consider  the singularly perturbed parabolic problem
\begin{subequations} \label{CproblemDiscontinuousBC}
\begin{align} \label{CproblemDiscontinuousBC-a}
&Lu =f(x,t) \ \text{in } Q,\qquad  u(1,t)=0, \ t\ge 0, \  u(x,0)=0, \ 0<x<1,
\end{align}
and the boundary condition at $x=0$ is given by
\begin{equation}
u(0,t)=\begin{cases}
\phi_1 (t), & \text{ if }  0\le t \le d, \\
\phi_2 (t), & \text{ if } d<t \le T,
\end{cases}
 \quad \phi_1(d^-) \ne \phi_2(d^+),\quad  \phi _1'(d^-)=  \phi _2'(d^+).
\end{equation}
Note that there is no loss in generality is assuming a homogenous initial condition. We assume that the following compatibility conditions are satisfied at $(0,0)$ and $(1,0)$:
\begin{align}
&\phi _1(0)=0, f(0,0) =\phi _1'(0^+), \ (f_t+f_{xx})(0,0) =\phi _1''(0^+)+b(0)\phi _1'(0^+), \label{comp3}\\
& f(1,0)= (f_t+f_{xx})(1,0) =0,
\end{align}
 and also the following regularity conditions
\begin{align}
& f,b \in C^{4+\gamma }(\bar Q ),  \phi _1 \in C^2(0,d), \phi _2 \in C^2(d,T) .
\end{align}
\end{subequations}
The discontinuous boundary condition on the left, will cause a singularity  to appear in the solution for $t\geq d$.

Decompose the solution of~\eqref{CproblemDiscontinuousBC} into the sum
\begin{equation}\label{decomp3}
u=[\phi](d)H(t-d)cs(x,t-d)+y, \quad [\phi](d):=\phi_2(d^+)- \phi_1(d^-),
\end{equation}
where $H(\cdot) $ is a unit step function defined by
$$
H(x):=
\begin{cases}
0, & \hbox{ for } x <0, \\
1, & \hbox{ for } x \geq 0.
\end{cases}
$$
Note that $cs(x,0)=0$. Observe that $y$ is the solution of the parabolic problem
\begin{subequations} \label{eq:ComponentY2DiscontinuousBC}
\begin{align}
&Ly =f+(b(d)-b(t))[\phi](d)H(t-d)cs(x,t-d)  \ \text{in } Q,\\
& y(1,t)=-[\phi](d)H(t-d)cs(1,t-d), \ t\ge 0, \\
& y(x,0)=\phi(x), \ 0<x<1,
\end{align}
 and
\begin{equation}
y(0,t)=u(0,t) -[\phi](d)H(t-d).
\end{equation}
\end{subequations}
 As in previous sections, we decompose $y$ into three subcomponents
\[
y = v +w_R + w_L,
\]
which are defined as the solutions of the following three parabolic problems.
\begin{subequations} \label{decomposition3}
The regular component satisfies
\begin{align}
&L^*v^* =f^* \ \text{in } Q ^*:=(-a,1+a) \times (0,T],\\
&v^*(-a,t)= v^*(1+a,t)=0, \ t\ge 0, \\
& v^*(x,0)=\phi^*(x), \ -a<x<1+a;
\end{align}
 where $\phi^*(x)$ is a smooth extension of the initial condition~\eqref{CproblemDiscontinuousBC-a}. The  right boundary layer component satisfies
\begin{align}
&Lw_R =0  \ \text{in } Q,\quad w_R(x,0)=0, \ 0<x<1,\\
& w_R(1,t)=y(1,t)-v(1,t), \ w_R(0,t)=0,\ t\ge 0;
\end{align}
 and the left  boundary layer component satisfies
\begin{align}
&Lw_L= Ly-f  \ \text{in } Q,\quad w_L(x,0)=0, \ 0<x<1,\\
& w_L(1,t)=0, \ w_L(0,t) = y(0,t)-v(0,t),  \ t\ge 0.
\end{align}
\end{subequations}
The regular component $v\in C^{4+\gamma}(\bar Q)$  and
since all time derivatives of $cs(1,t-d)$ are zero at $t=d$, we have that $w_R\in C^{4+\gamma}(\bar Q)$.
In addition, $w_L\in C^{2+\gamma}(\bar Q)$.
Hence, the character of the function $y$ for Problem Class 3 is the same as for Problem Class 1. In other words, the bounds on the derivatives of the components of $y$ given in Theorem 1 also apply in the case of Problem Class 3. However, the character of the singular component $u-y$ is different for the two problem classes.

\section{Numerical Method} \label{sec:numericalmethod}

For all three problem classes we employ a classical  finite difference operator (backward Euler in time  and standard central differences in space) on an appropriate mesh (which will be piecewise-uniform in space and uniform in time). The piecewise-uniform Shishkin mesh for each of the three Problem Classes  $n, (n=1,2,3) $ will be denoted by $ \bar Q_n ^{N,M}$. The  numerical method
{\footnote{ The finite difference operators are defined by:
\begin{eqnarray*} D^+_x U(x_i,t _j):= D^-_x U(x_{i+1},t _j);\quad D^-_x U(x_i,t _j):=\frac{U(x_{i},t _j) - U(x_{i-1},t_{j})}{h_i}, \\
D^-_t U(x_i,t _j) := \frac{U(x_i,t _j) - U(x_{i},t_{j-1})}{k_j }, \quad \delta ^2 _x U(x_i,t_j):= \frac{(D^+_x - D^-_x)U(x_{i},t _j)}{\hbar_i}
\end{eqnarray*}
and the mesh steps are $h_i:=x_{i}-x_{i-1}, \hbar_i =(h_{i+1}+h_i)/2,\quad k=k_j:= t_j - t_{j-1}$.}}:
\begin{subequations} \label{discrete-problem}
\begin{eqnarray}
L^{N,M} Y(x_i,t_j) =f(x_i,t_j), \qquad (x_i,t_j)\in  { Q_n^{N,M},} \\
Y(x_i,t_j)=y(x_i,t_j) , \qquad (x_i,t_j)\in { \partial Q _n^{N,M},} \\
\hbox{where} \quad L^{N,M} Y(x_i,t_j) :=  ( -\ve\delta^2_x +b(t_j) I+D^-_t)Y(x_i,t_j).
\end{eqnarray}
\end{subequations}

For Problem Class 1, the Shishkin mesh $\bar Q _1^{N,M}$ is  defined via :
$$
[0,1]=[0,\sigma ]\cup[\sigma,1-\sigma]\cup[1-\sigma,1],
\quad
\sigma :=\min \left\{\frac1{4},\frac{4}{\mu}\sqrt{\vr T} \ln N \right\}, \quad \mu < 1;
$$
and $N/4,N/2,N/4$ grid points are uniformly distributed in each subinterval, respectively.
For Problem Class 3, the Shishkin mesh  is $ \bar Q _3^{N,M} = \bar Q _1^{N,M}$.

For Problem Class 2, the Shishkin mesh  is  $\bar Q _2^{N,M}$, which is defined via
$$
 [0,1]=[0,\tau]\cup[\tau,d-\tau]\cup[d-\tau,d+\tau]\cup[d+\tau,1-\tau]\cup[1-\tau,1],
$$
with
$$
\tau: =\min \left\{\frac1{8},\frac{4}{\mu}\sqrt{\vr T} \ln N \right\}, \quad \mu < 1
$$
and $N/8,N/4,N/4,N/4,N/8$ grid points are uniformly distributed in each subinterval, respectively.
 Although it is required that $\mu<1$ in the theoretical error analysis, in the numerical results section, we simply have taken $\mu =1$.

\begin{thm} \label{th:main}
 Let be  $Y$ the solution of the finite difference scheme~\eqref{discrete-problem} and $y$ the solution of the  continuous problem. Then,   the global approximation $\bar Y$ on $\bar Q$ generated by the values of $Y$ on $\bar Q^{N,M}_n$ and bilinear interpolation, satisfies
\begin{equation} \label{eq:GlobalError}
 \Vert y-\bar Y \Vert _{\bar Q} \leq  (CN^{-2} \ln^2 N  + CM^{-1}) \ln M,
\end{equation}
for each of the three Problem Classes (\ref{Cproblem}), (\ref{CproblemDiscontinuous}) and (\ref{CproblemDiscontinuousBC}).
\end{thm}

\begin{proof}    For each of the three Problem Classes, the discrete solution $Y$ is decomposed along the same lines as its continuous counterpart $y$.

Let us first consider  Problem  Classes 1 and 3. Using the bounds on the derivatives of the components in Theorem 1, truncation error bounds, discrete maximum principle and a suitable discrete barrier function and following the arguments in \cite{ria}, we can establish the following bounds
\begin{eqnarray*}
\Vert v- V \Vert _{\bar Q^{N,M}}, \Vert w_R- W_R \Vert _{\bar Q^{N,M}} \leq  C N^{-2} \ln^2 N  + CM^{-1}.
\end{eqnarray*}
It remains to bound the error due to the left boundary layer component.  We introduce the following notation for this error and the associated truncation error
\[
E^j_i := (w_L-W_L)(x_i,t_j) \quad \hbox{and} \quad {\cal T}_{i,j}:= L^{N,M} E^j_i.
\]
Note that
\begin{align*}
 \vert \delta _x^2 w_L(x_i,t_j) \vert & \leq C \left\Vert \frac{\partial ^2 w_L}{\partial x^2} \right\Vert _{(x_{i-1},x_{i+1}) \times \{ t_j \} },  \\
 \vert D_t^- w_L(x_i,t_j) \vert  & \leq C \left \Vert \frac{\partial  w_L}{\partial t} \right \Vert _{ \{ x_i \} \times (t_{j-1}, t_j )}.
\end{align*}
Hence, using the bounds (\ref{expLlayer-xbounds}) on the derivatives of $w_L$, we have that outside the left layer
\[
\vert {\cal T}_{i,j} \vert \leq CN^{-2}, \quad x_i \geq \sigma , t_j >0.
\]
Within the left layer, using the bounds (\ref{expLlayer-4th-xbounds}), (\ref{expLlayer-2nd-tbounds}) on the higher derivatives of $w_L$, we have the truncation error bounds
\begin{eqnarray*}
\vert {\cal T}_{i,1} \vert &\leq& C\frac{(N^{-1}\ln N)^2}{t_1} + C \left \Vert \frac{\partial w_L}{\partial t} \right \Vert _{ \{ x_i \} \times (0, t_1 )} \leq  C\frac{(N^{-1}\ln N)^2}{t_1} + C, \quad x_i < \sigma \\
\vert {\cal T}_{i,j} \vert &\leq& C\frac{(N^{-1}\ln N)^2}{t_j}  + C k \left \Vert \frac{\partial^2 w_L}{\partial t^2} \right \Vert _{ \{ x_i \} \times (t_{j-1}, t_j )}  \\
&\leq& C\frac{(N^{-1}\ln N)^2}{t_j} +C\frac{k}{t_{j-1}}, \quad x_i < \sigma , t_j > t_1.
\end{eqnarray*}
Hence, at all time levels, we have  the truncation error bound
\[
\vert {\cal T}_{i,j} \vert \leq C\frac{(N^{-1}\ln N)^2 +M^{-1}}{t_j}, \quad x_i < \sigma , t_j \geq t_1.
\]
We now mimic  the argument in \cite{Zhemukhov1} and note that at each time level,
\[
-\ve \delta ^2_x  E^j_i + \left(b(x_i,t_j) +\frac{1}{k}\right)  E^j_i =  {\cal T}_{i,j} + \frac{1}{k}  E^{j-1}_i, \ t_j >0 .
\]
From this we can deduce the error bound
\begin{eqnarray} \label{ErrorZ}
\vert  E^j_i \vert &\leq& C k \sum _{n=1}^j \vert {\cal T}_{i,n} \vert \leq C \left( (N^{-1}\ln N)^2 +M^{-1}\right) \sum _{n=1}^j \frac{1}{n} \nonumber\\
&\leq&  C \left((N^{-1}\ln N)^2 +M^{-1}\right) \left (1+ \int _{s=1}^{j} \frac{ds}{s} \right) \nonumber\\
&\leq&   C((N^{-1}\ln N)^2 +M^{-1})\ln  (1+j).
\end{eqnarray}
In the case of Problem Class 2, we have an additional interior layer component $w_I$. The bounding of the error $\Vert w_I- W_I \Vert$ follows the same argument as above.

One can extend this nodal error bound to a global error bound by applying the argument in \cite[pp. 56-57]{fhmos} and using the modification in \cite{go_dbb} to manage the initial singularity.
\end{proof}

\section{Numerical Results} \label{sec:numerical}
The orders of convergence of the finite difference scheme~\eqref{discrete-problem} are estimated using the two-mesh principle~\cite{fhmos}. We denote by  $Y^{N,M}$ and $Y^{2N,2M}$ the computed solutions with~\eqref{discrete-problem} on the Shishkin meshes $ Q^{N,M}_n$ and $Q^{2N,2M}_n$, respectively. These solutions are used to computed the maximum two-mesh global differences
$$
D^{N,M}_\ve:= \Vert \bar Y^{N,M}-\bar Y^{2N,2M}\Vert _{ Q^{N,M}_n \cup Q^{2N,2M}_n}
$$
where $\bar Y^{N,M}$ and $\bar Y^{2N,2M}$ denote the bilinear interpolation of the discrete solutions $Y^{N,M}$ and $Y^{2N,2M}$ on the mesh $ Q^{N,M}_n \cup Q^{2N,2M}_n$. Then, the orders of global  convergence $P^{N,M}_\ve$ are estimated in a standard way \cite{fhmos}
$$
P^{N,M}_\ve:=  \log_2\left (\frac{D^{N,M}_\ve}{D^{2N,2M}_\ve} \right).
$$
The uniform  two-mesh global differences $D^{N,M}$ and their corresponding uniform orders of global convergence $P^{N,M}$ are calculated by
$$
D^{N,M}:= \max_{\ve \in S} D^{N,M}_\ve, \quad P^{N,M}:=  \log_2\left (\frac{D^{N,M}}{D^{2N,2M}} \right),
$$
where $S=\{2^0,2^{-1},\ldots,2^{-30}\}$.

In order that the temporal discretization error dominates the spatial discretization error, in all the tables of this paper, except in Table~\ref{tb:NoDBBComponentYexGlobalN=M}, we have taken $N=2^4\times M$.

\subsection{Problem Class 1}

 We present numerical results for two examples from this first class of problems. In the first example the coefficient of the reaction term depends only on the temporal variable;  while in the second example, it depends on the spatial variable. The numerical results computed with the analytical/numerical method of this paper suggest that the method is uniformly and globally convergent in both cases.

\begin{example} \label{Section:Ex1}
Consider   problem (\ref{Cproblem}), with the data given by
\begin{equation}\label{ex}
 b(x,t)= 1+t, \ f(x,t)=4x(1-x)t+t^2, \ \phi(x)=(1-x)^2.
\end{equation}
\end{example}
The maximum two-mesh global differences associated with the component $y$ and the orders of convergence are given in Table~\ref{tb:NoDBBComponentYexGlobal}. Observe that the numerical results  show that the method is  first-order globally parameter-uniformly convergent.   In Table~\ref{tb:NoDBBComponentYexGlobalN=M}, we give the uniform two-mesh global differences taking $N=M$ and the computed orders of convergence illustrate that the method is almost second order convergent; in this case the spatial discretization errors dominate the temporal discretization errors. The numerical results in Tables~\ref{tb:NoDBBComponentYexGlobal} and~\ref{tb:NoDBBComponentYexGlobalN=M} are in agreement with our error estimates in Theorem~\ref{th:main}.
\begin{table}[h]
\caption{Example~\ref{Section:Ex1} from Problem Class 1: Maximum two-mesh global differences and orders of convergence for the function $y$ in~\eqref{eq:ComponentY2} using a piecewise uniform Shishkin mesh}
\begin{center}{\tiny \label{tb:NoDBBComponentYexGlobal}
\begin{tabular}{|c||c|c|c|c|c|}
 \hline  & N=256,M=16 & N=512,M=32 & N=1024,M=64 & N=2048,M=128 & N=4096,M=256 \\
\hline \hline  $\vr=2^{0}$
 &{\bf 1.295E-02} &{\bf 6.990E-03} &{\bf 3.650E-03 }&{\bf 1.870E-03 }&{\bf 9.453E-04} \\
&0.890&0.938&0.965&0.984&
\\ \hline $\vr=2^{-2}$
&3.789E-03 &1.980E-03 &1.013E-03 &5.128E-04 &2.580E-04 \\
&0.936&0.966&0.983&0.991&
\\ \hline $\vr=2^{-4}$
&2.214E-03 &1.081E-03 &5.339E-04 &2.653E-04 &1.322E-04 \\
&1.035&1.017&1.009&1.005&
\\ \hline $\vr=2^{-6}$
&4.216E-03 &2.083E-03 &1.035E-03 &5.161E-04 &2.577E-04 \\
&1.017&1.008&1.004&1.002&
\\ \hline $\vr=2^{-8}$
&4.971E-03 &2.456E-03 &1.220E-03 &6.084E-04 &3.037E-04 \\
&1.017&1.009&1.004&1.002&
\\ \hline $\vr=2^{-10}$
&5.269E-03 &2.598E-03 &1.290E-03 &6.428E-04 &3.208E-04 \\
&1.020&1.010&1.005&1.003&
\\ \hline $\vr=2^{-12}$
&6.402E-03 &2.668E-03 &1.321E-03 &6.569E-04 &3.276E-04 \\
&1.263&1.015&1.008&1.004&
\\ \hline $\vr=2^{-14}$
&1.092E-02 &4.011E-03 &1.342E-03 &6.655E-04 &3.313E-04 \\
&1.445&1.580&1.012&1.006&
\\ \hline $\vr=2^{-16}$
&1.092E-02 &4.013E-03 &1.347E-03 &6.685E-04 &3.326E-04 \\
&1.445&1.575&1.011&1.007&
\\ \hline .&.&.&.&.&.\\.&.&.&.&.&.\\.&.&.&.&.&.
\\ \hline $\vr=2^{-30}$
&1.093E-02 &4.014E-03 &1.352E-03 &6.707E-04 &3.337E-04 \\
&1.445&1.571&1.011&1.007&
\\ \hline $D^{N,M}$
&1.295E-02 &6.990E-03 &3.650E-03 &1.870E-03 &9.453E-04 \\
$P^{N,M}$ &0.890&0.938&0.965&0.984&\\ \hline \hline
\end{tabular}}
\end{center}
\end{table}

\begin{table}[h]
\caption{ Example~\ref{Section:Ex1} from Problem Class 1: Uniform two-mesh global differences and orders of convergence for the function $y$ in~\eqref{eq:ComponentY2} using a piecewise uniform Shishkin mesh with $N=M$}
\begin{center}{\tiny \label{tb:NoDBBComponentYexGlobalN=M}
\begin{tabular}{|c||c|c|c|c|c|}
 \hline  & N=M=64 & N=M=128 & N=M=256 & N=M=512 & N=M=1024 \\
\hline $D^{N,M}$
&4.972E-02 &2.548E-02 &1.117E-02 &3.983E-03 &1.330E-03 \\
$P^{N,M}$ &0.964&1.189&1.488&1.583&\\ \hline \hline
\end{tabular}}
\end{center}
\end{table}

We display  in Figure~\ref{fig:NoDBBY-bl} 
 the numerical approximation to the function $y$ defined in~\eqref{eq:ComponentY2},  which exhibits only boundary layers.
The numerical solution to problem~\eqref{Cproblem}-\eqref{ex} is displayed in Figure~\ref{fig:NoDBBU-bl}, which exhibits  both  boundary layers and the singularity caused by the incompatibility between the  initial and boundary conditions.
  \begin{figure}[h!]
\centering
\resizebox{\linewidth}{!}{
	\begin{subfigure}[Entire domain $\bar Q$]{
		\includegraphics[scale=0.5, angle=0]{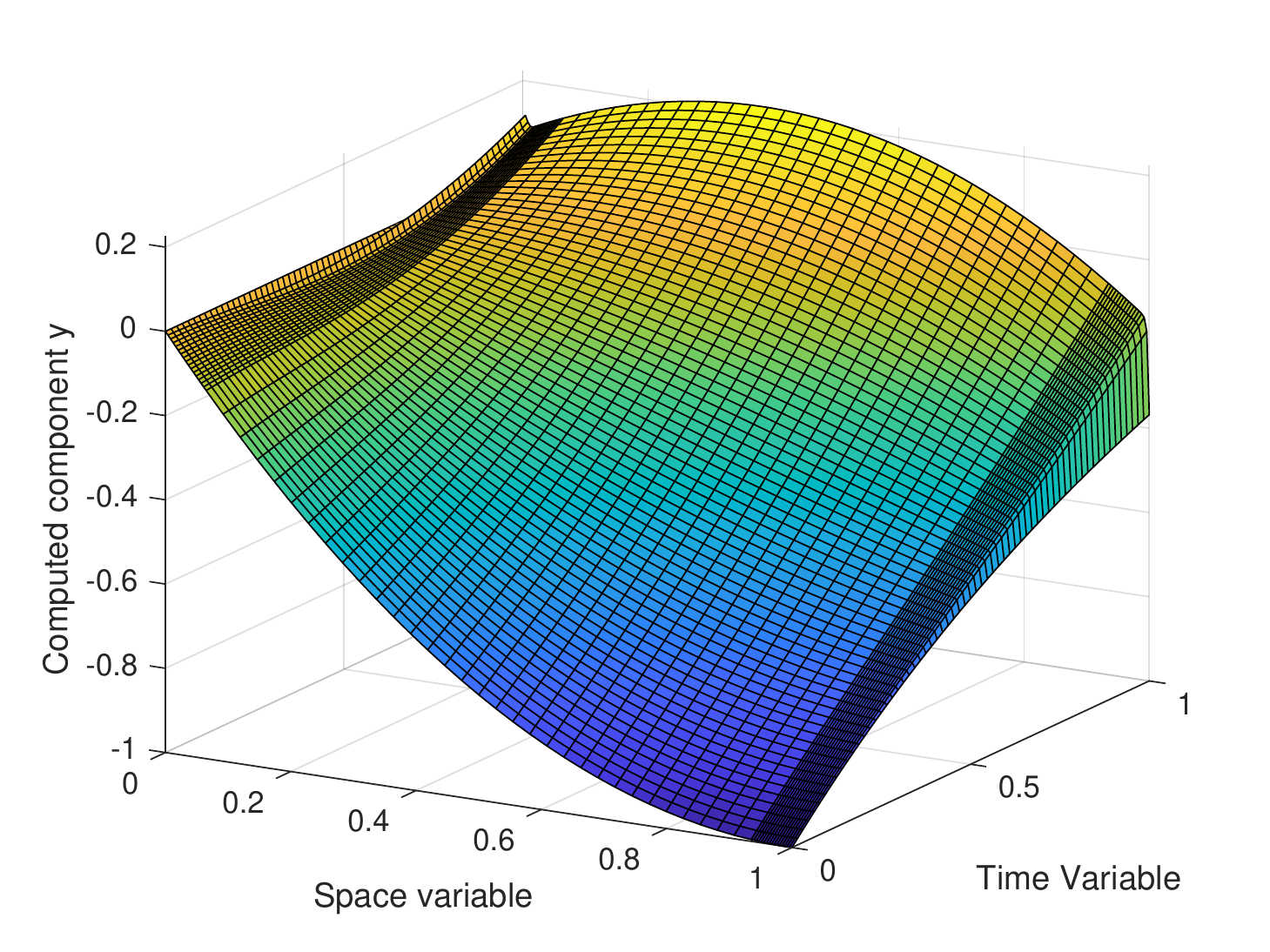}
		}
    \end{subfigure}
\begin{subfigure}[A zoom in on the corner $(0,0)$]{
		\includegraphics[scale=0.5, angle=0]{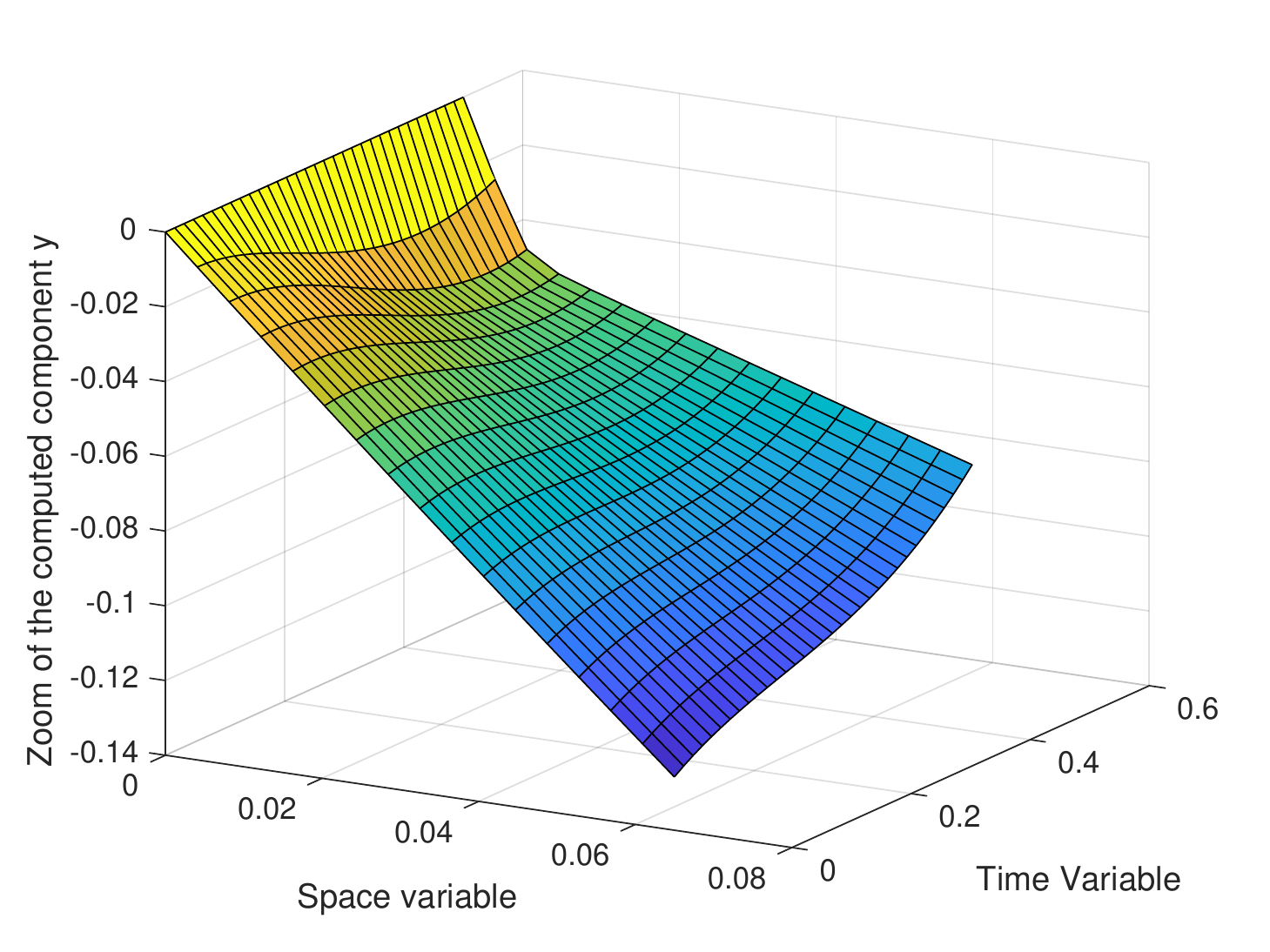}
		}
	\end{subfigure}
}
	\caption{ Example~\ref{Section:Ex1} from Problem Class 1: The numerical approximation to $y$ with $\vr =2^{-16}$ and $N=M=64$}
	\label{fig:NoDBBY-bl}
 \end{figure}

 \begin{figure}[h!]
\centering
\resizebox{\linewidth}{!}{
	\begin{subfigure}[Entire domain $\bar Q$]{
		\includegraphics[scale=0.5, angle=0]{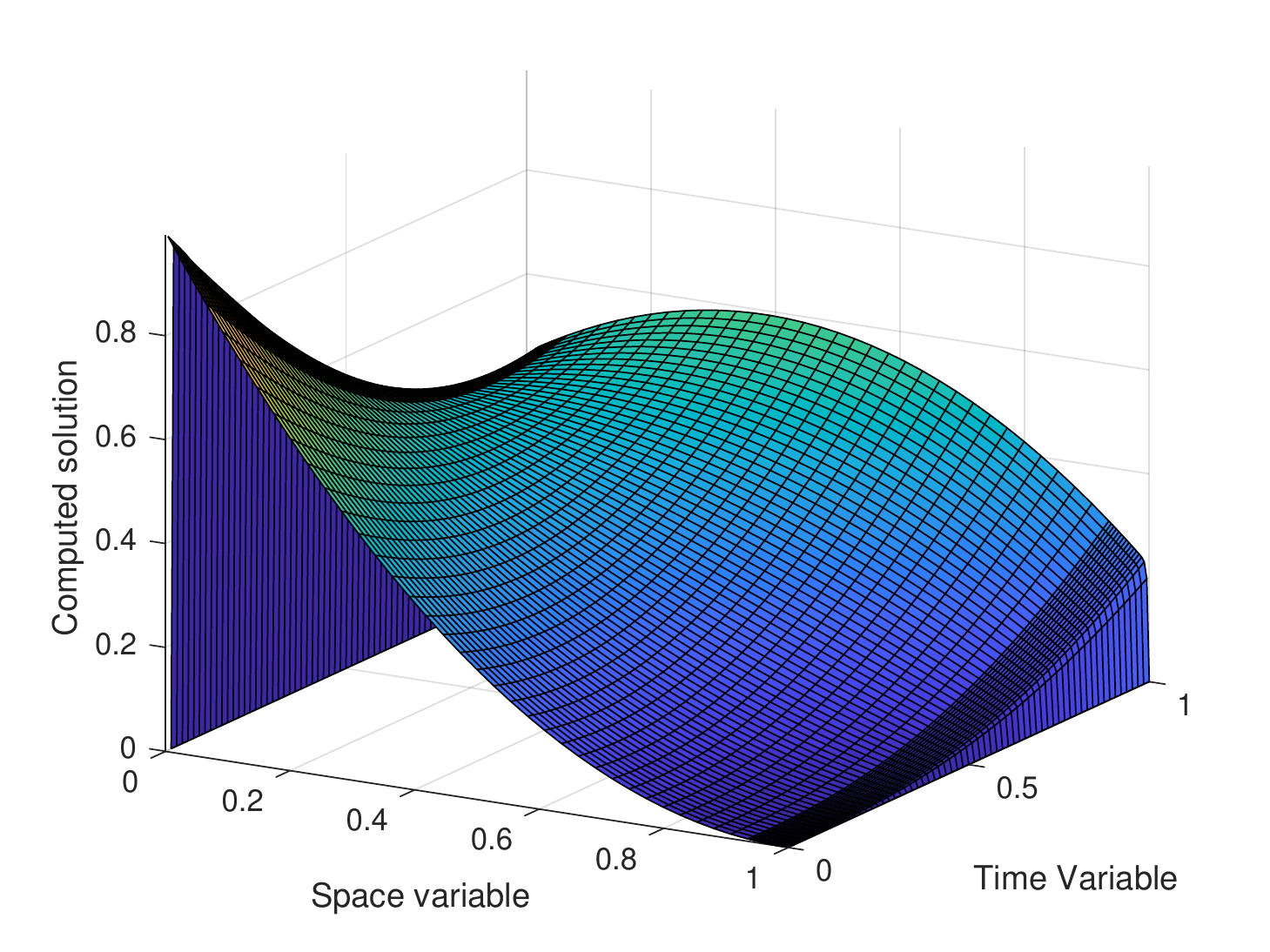}
		}
    \end{subfigure}
\begin{subfigure}[A zoom in on the corner $(0,0)$]{
		\includegraphics[scale=0.5, angle=0]{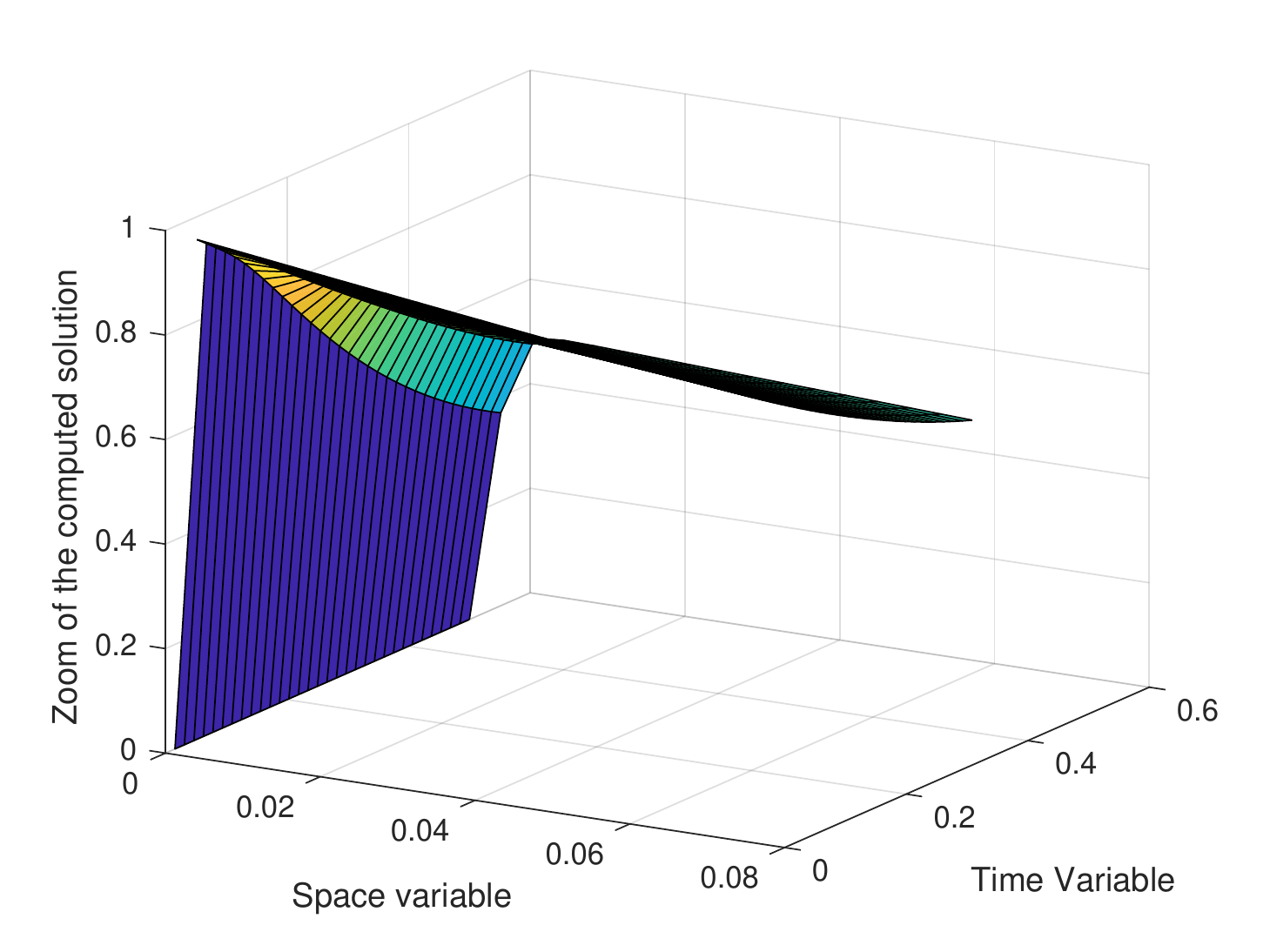}
		}
	\end{subfigure}
}
	\caption{ Example~\ref{Section:Ex1} from Problem Class 1: The  approximation $ s+Y$ to the solution $u$ with $\vr =2^{-16}$ and $N=M=64$}
	\label{fig:NoDBBU-bl}
 \end{figure}

\begin{example} \label{Section:Ex2}
Consider  problem (\ref{Cproblem}), with the data given by
\begin{equation}\label{ex2}
b(x,t)= 1+10x, \quad f(x,t)=4x(1-x)t+t^2, \quad \phi(x)=(1-x)^2.
\end{equation}
\end{example}
The maximum two-mesh global differences associated with the component $y$ and the orders of convergence are given in Table~\ref{tb:NoDBBComponentYex2Global}. Observe that the numerical results indicate that the method is globally  parameter-uniformly convergent. Comparing these orders of convergence with those in Table~\ref{tb:NoDBBComponentYexGlobal}, we see that the theoretical assumption of $b$ being independent of the space variable appears to be not necessary in order to observe parameter-uniform convergence.

\begin{table}[h]
\caption{ Example~\ref{Section:Ex2} from Problem Class 1: Maximum two-mesh global differences and orders of convergence for the function $y$ in~\eqref{eq:ComponentY2} using a piecewise uniform Shishkin mesh}
\begin{center}{\tiny \label{tb:NoDBBComponentYex2Global}
\begin{tabular}{|c||c|c|c|c|c|}
 \hline  & N=256,M=16 & N=512,M=32 & N=1024,M=64 & N=2048,M=128 & N=4096,M=256 \\
\hline \hline  $\vr=2^{0}$
&4.837E-03 &4.267E-03 &2.321E-03 &1.160E-03 &5.823E-04 \\
&0.181&0.878&1.000&0.995&
\\ \hline $\vr=2^{-2}$
&9.114E-03 &4.665E-03 &2.341E-03 &1.172E-03 &5.863E-04 \\
&0.966&0.994&0.998&0.999&
\\ \hline $\vr=2^{-4}$
&1.086E-02 &5.523E-03 &{\bf 2.787E-03} &{\bf 1.400E-03} &{\bf 7.016E-04} \\
&0.975&0.987&0.993&0.997&
\\ \hline $\vr=2^{-5}$
&{\bf 1.092E-02} &{\bf 5.531E-03} &2.784E-03 &1.398E-03 &7.006E-04 \\
&0.982&0.990&0.994&0.997&
\\ \hline $\vr=2^{-6}$
&1.068E-02 &5.387E-03 &2.712E-03 &1.361E-03 &6.814E-04 \\
&0.988&0.990&0.995&0.998&
\\ \hline $\vr=2^{-8}$
&1.047E-02 &5.305E-03 &2.672E-03 &1.341E-03 &6.717E-04 \\
&0.980&0.990&0.995&0.997&
\\ \hline $\vr=2^{-10}$
&1.056E-02 &5.349E-03 &2.693E-03 &1.351E-03 &6.769E-04 \\
&0.982&0.990&0.995&0.997&
\\ \hline $\vr=2^{-12}$
&1.059E-02 &5.361E-03 &2.698E-03 &1.354E-03 &6.782E-04 \\
&0.982&0.990&0.995&0.997&
\\ \hline $\vr=2^{-14}$
&1.060E-02 &5.365E-03 &2.700E-03 &1.355E-03 &6.786E-04 \\
&0.982&0.991&0.995&0.997&
\\ \hline $\vr=2^{-16}$
&1.061E-02 &5.368E-03 &2.701E-03 &1.355E-03 &6.787E-04 \\
&0.983&0.991&0.995&0.997&
\\ \hline .&.&.&.&.&.\\.&.&.&.&.&.\\.&.&.&.&.&.
\\ \hline $\vr=2^{-30}$
&1.062E-02 &5.371E-03 &2.702E-03 &1.355E-03 &6.788E-04 \\
&0.984&0.991&0.995&0.998&
\\ \hline $D^{N,M}$
&1.092E-02 &5.531E-03 &2.787E-03 &1.400E-03 &7.016E-04 \\
$P^{N,M}$ &0.982&0.989&0.993&0.997&\\ \hline \hline
\end{tabular}}
\end{center}
\end{table}

\subsection{Problem Class 2}

\begin{example} \label{ex:exDiscontinuoust}
Consider   problem (\ref{CproblemDiscontinuous}) with the data given by
\begin{subequations} \label{exDiscontinuoust}
\begin{align}
&  b(x,t)= 1+10xt, \quad f(x,t)=4 x (1-x)t+t^2, \\
&  \phi(x)=
\begin{cases}
-1+(2x-1)^2, & \quad \text{ if } \quad   0\le x \le 0.5, \\
1-(2x-1)^2, & \quad \text{ if } \quad  0.5<x \le 1.
\end{cases}
\end{align}
\end{subequations}
Observe that in  this example the coefficient  $b$ depends on the temporal and spatial variables and, moreover, $\phi '(0.5^+) = \phi '(0.5^-)$ but $\phi ''(0.5^+) \neq \phi ''(0.5^-)$. The schemes considered here to approximate the solution of this example are  defined on the Shishkin mesh  $ \bar Q _2^{N,M}$.
\end{example}
 If the singularity is not stripped off and Example~\ref{ex:exDiscontinuoust} is simply solved with backward Euler method and standard central finite differences on the Shishkin mesh $\bar{Q}^{N,M}_2$, the method is not globally convergent for any value of $\vr$. This is illustrated in Table~\ref{tb:DiscontinuousUtGlobal-NoDecomp} where the uniform two-mesh global differences are given.

\begin{table}[h]
\caption{Example~\ref{ex:exDiscontinuoust} from Problem Class 2: Maximum two-mesh global differences and orders of convergence for $u$ using a piecewise uniform Shishkin mesh, without separating off the singularity}
\begin{center}{\tiny \label{tb:DiscontinuousUtGlobal-NoDecomp}
\begin{tabular}{|c||c|c|c|c|c|}
 \hline  & N=256,M=16 & N=512,M=32 & N=1024,M=64 & N=2048,M=128 & N=4096,M=256
\\ \hline $D^{N,M}$
&6.698E-01 &5.707E-01 &4.992E-01 &4.994E-01 &4.996E-01 \\
$P^{N,M}$ &0.231&0.193&-0.001&-0.001&\\ \hline \hline
\end{tabular}}
\end{center}
\end{table}

 We show now the numerical results when the singularity is stripped off. The maximum two-mesh global differences associated with the component $y$ and the orders of convergence are given in Table~\ref{tb:DiscontinuousYtGlobal}. Observe that the numerical results indicate that the method is globally  and uniformly convergent.

\begin{table}[h]
\caption{Example~\ref{ex:exDiscontinuoust} from Problem Class 2: Maximum two-mesh global differences and orders of convergence for the function $y$
in~\eqref{eq:ComponentY2Discontinuous} using a piecewise uniform Shishkin mesh}
\begin{center}{\tiny \label{tb:DiscontinuousYtGlobal}
\begin{tabular}{|c||c|c|c|c|c|}
 \hline  & N=256,M=16 & N=512,M=32 & N=1024,M=64 & N=2048,M=128 & N=4096,M=256 \\
\hline \hline  $\vr=2^{0}$
&1.683E-02 &8.549E-03 &4.277E-03 &{\bf 2.134E-03} &{\bf 1.066E-03} \\
&0.978&0.999&1.003&1.001&
\\ \hline $\vr=2^{-2}$
&6.557E-03 &3.177E-03 &1.563E-03 &7.741E-04 &3.852E-04 \\
&1.045&1.024&1.014&1.007&
\\ \hline $\vr=2^{-4}$
&5.748E-03 &2.992E-03 &1.527E-03 &7.717E-04 &3.879E-04 \\
&0.942&0.970&0.985&0.992&
\\ \hline $\vr=2^{-6}$
&8.330E-03 &4.346E-03 &2.220E-03 &1.122E-03 &5.642E-04 \\
&0.939&0.969&0.984&0.992&
\\ \hline $\vr=2^{-8}$
&9.535E-03 &4.981E-03 &2.546E-03 &1.287E-03 &6.469E-04 \\
&0.937&0.968&0.984&0.992&
\\ \hline $\vr=2^{-10}$
&1.128E-02 &5.618E-03 &2.804E-03 &1.401E-03 &6.999E-04 \\
&1.005&1.003&1.001&1.001&
\\ \hline $\vr=2^{-12}$
&1.245E-02 &6.212E-03 &3.103E-03 &1.551E-03 &7.751E-04 \\
&1.003&1.001&1.001&1.000&
\\ \hline $\vr=2^{-14}$
&1.880E-02 &6.559E-03 &3.278E-03 &1.639E-03 &8.194E-04 \\
&1.519&1.000&1.000&1.000&
\\ \hline $\vr=2^{-15}$
&{\bf 3.134E-02} &1.104E-02 &3.335E-03 &1.667E-03 &8.338E-04 \\
&1.505&1.727&1.000&1.000&
\\ \hline $\vr=2^{-16}$
&2.964E-02 &{\bf 1.266E-02} &{\bf 4.588E-03 }&1.689E-03 &8.445E-04 \\
&1.228&1.464&1.442&1.000&
 \\ \hline .&.&.&.&.&.\\.&.&.&.&.&.\\.&.&.&.&.&.
\\ \hline $\vr=2^{-30}$
&2.957E-02 &1.264E-02 &4.584E-03 &1.752E-03 &8.755E-04 \\
&1.226&1.463&1.388&1.001&
\\ \hline $D^{N,M}$
&3.134E-02 &1.266E-02 &4.588E-03 &2.134E-03 &1.066E-03 \\
$P^{N,M}$ &1.308&1.464&1.104&1.001&\\ \hline \hline
\end{tabular}}
\end{center}
\end{table}

 Figure~\ref{fig:DiscontinuousYt} 
displays both the numerical approximation to the function $y$ defined in~\eqref{eq:ComponentY2Discontinuous}
and the numerical solution to  problem~\eqref{CproblemDiscontinuous} and~\eqref{exDiscontinuoust} is displayed in Figure~\ref{fig:DiscontinuousYt}. The presence of an interior layer is evident in both figures. 

 \begin{figure}[h!]
\centering
\resizebox{\linewidth}{!}{
	\begin{subfigure}[The computed $Y$]{
		\includegraphics[scale=0.5, angle=0]{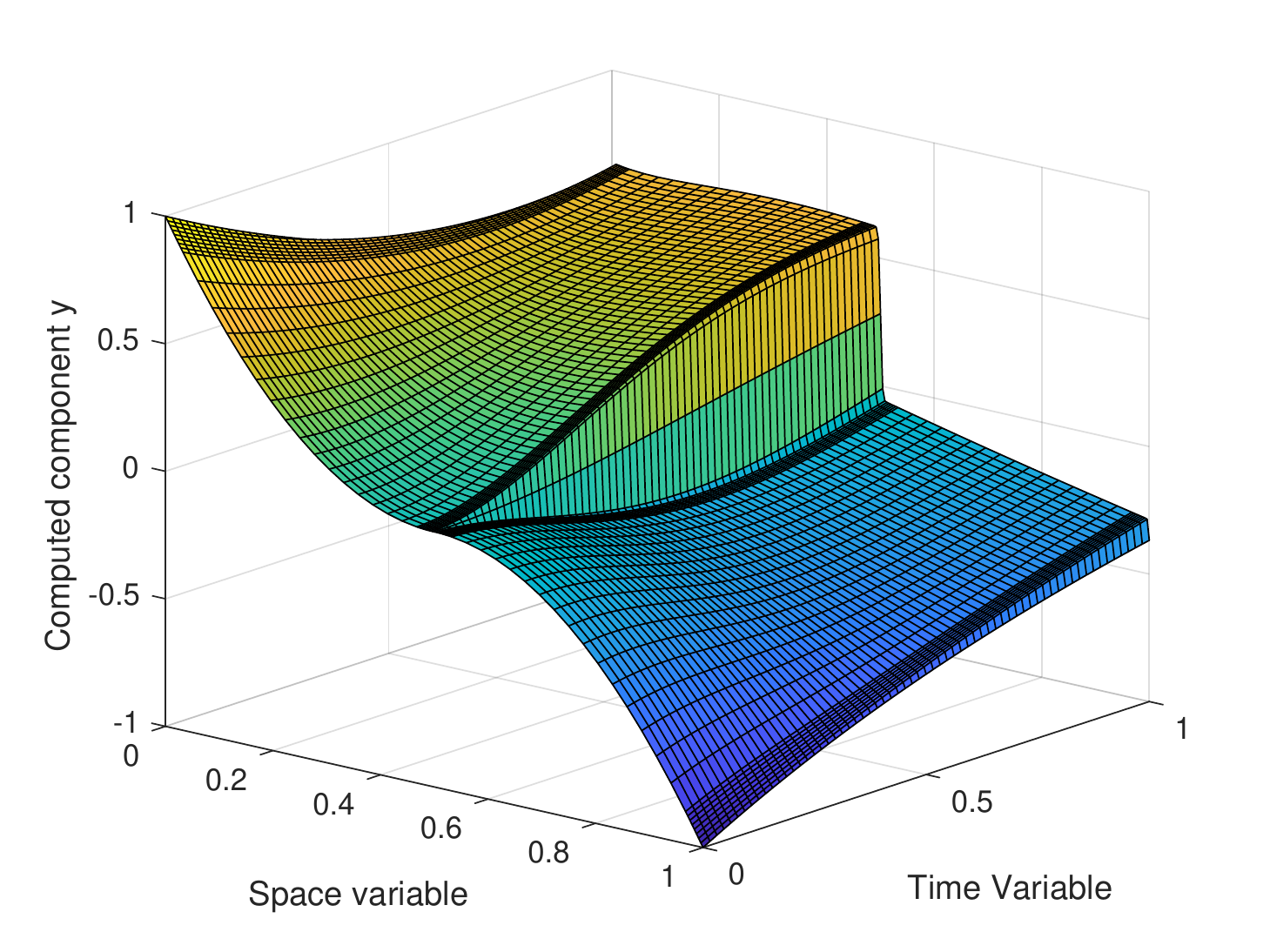}
		}
    \end{subfigure}
\begin{subfigure}[The  approximation $ s+Y$]{
		\includegraphics[scale=0.5, angle=0]{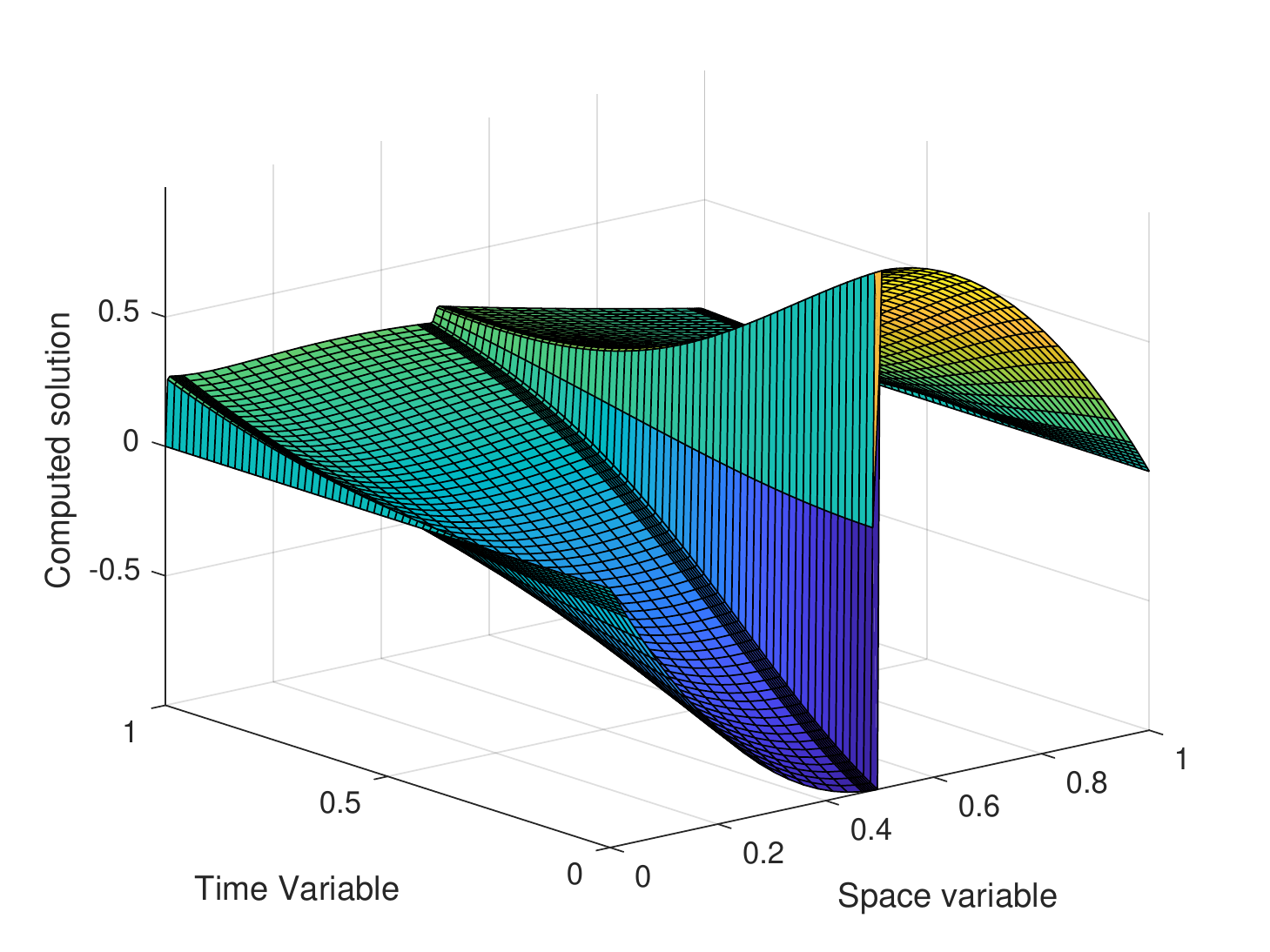}
		}
	\end{subfigure}
}
	\caption{Example~\ref{ex:exDiscontinuoust} from Problem Class 2: The numerical approximation to $y$ and the approximation  $s+ Y$ to the solution $u$, with $\vr =2^{-16}$ and $N=M=64$}
	\label{fig:DiscontinuousYt}
 \end{figure}

\subsection{Problem Class 3}

\begin{example} \label{ex:exDiscontinuousBC}
Consider  the problem (\ref{CproblemDiscontinuousBC}), with the data taken to be
\begin{equation}\label{exDiscontinuousBC}
b(x,t)= 1+x, \quad f(x,t)=4(1+x)(1-x)t+t^2,
\end{equation}
and
$$
u(0,t)=\begin{cases}
0,   & \text{ if } \  0\le t \le 0.25, \\
0.5, & \text{ if } \  0.25<t \le 1.
\end{cases}
$$
 Observe that in this example the function $b=b(x)$. The schemes considered here to approximate the solution  are  defined on the Shishkin mesh  $ \bar Q _3^{N,M}$.
\end{example}

Once again, we first confirm the need to use our analytical/numerical approach to approximate the Problem Class 3. If we use backward Euler method and standard central finite differences on the Shishkin mesh $\bar{Q}^{N,M}_3$ to approximate  Example~\ref{ex:exDiscontinuousBC} without separating off the singularity, it is not globally convergent for any value of $\vr$. By way of illustration, the uniform two-mesh global differences are given in Table~\ref{tb:DiscontinuousBCUGlobal-NoDecomp}.

\begin{table}[h]
\caption{Example~\ref{ex:exDiscontinuousBC} from Problem Class 3: Maximum two-mesh global differences and orders of convergence for $u$ using a piecewise uniform Shishkin mesh, without separating off the singularity}
\begin{center}{\tiny \label{tb:DiscontinuousBCUGlobal-NoDecomp}
\begin{tabular}{|c||c|c|c|c|c|}
 \hline  & N=256,M=16 & N=512,M=32 & N=1024,M=64 & N=2048,M=128 & N=4096,M=256
\\ \hline $D^{N,M}$
&2.500E-01 &2.500E-01 &2.500E-01 &2.500E-01 &2.500E-01 \\
$P^{N,M}$ &0.000&0.000&0.000&0.000&\\ \hline \hline
\end{tabular}}
\end{center}
\end{table}

 We show now the numerical results when the singularity is stripped off. The maximum two-mesh global differences associated with the component $y$ and the orders of convergence are given in Table~\ref{tb:DiscontinuousYGlobalBC}. Observe that the numerical results indicate that the method is globally parameter-uniformly convergent.
\begin{table}[h]
\caption{Example~\ref{ex:exDiscontinuousBC} from Problem Class 3: Maximum two-mesh global differences and orders of convergence for the function $y$ in~\eqref{eq:ComponentY2DiscontinuousBC} using a piecewise uniform Shishkin mesh}
\begin{center}{\tiny \label{tb:DiscontinuousYGlobalBC}
\begin{tabular}{|c||c|c|c|c|c|}
 \hline  & N=256,M=16 & N=512,M=32 & N=1024,M=64 & N=2048,M=128 & N=4096,M=256 \\
\hline \hline  $\vr=2^{0}$
&2.765E-03 &1.478E-03 &7.901E-04 &4.057E-04 &2.058E-04 \\
&0.903&0.904&0.962&0.979&
\\ \hline $\vr=2^{-2}$
&4.421E-03 &2.250E-03 &1.136E-03 &5.714E-04 &2.866E-04 \\
&0.975&0.985&0.992&0.996&
\\ \hline $\vr=2^{-4}$
&7.926E-03 &3.909E-03 &1.940E-03 &9.664E-04 &4.823E-04 \\
&1.020&1.010&1.005&1.003&
\\ \hline $\vr=2^{-6}$
&1.022E-02 &5.079E-03 &2.532E-03 &1.264E-03 &6.315E-04 \\
&1.009&1.004&1.002&1.001&
\\ \hline $\vr=2^{-8}$
&1.050E-02 &5.214E-03 &2.598E-03 &1.297E-03 &6.478E-04 \\
&1.010&1.005&1.003&1.001&
\\ \hline $\vr=2^{-10}$
&1.059E-02 &5.260E-03 &2.621E-03 &1.308E-03 &6.535E-04 \\
&1.010&1.005&1.003&1.001&
\\ \hline $\vr=2^{-12}$
&1.062E-02 &5.275E-03 &2.628E-03 &1.312E-03 &6.553E-04 \\
&1.010&1.005&1.003&1.001&
\\ \hline $\vr=2^{-14}$
&1.063E-02 &5.278E-03 &2.630E-03 &1.313E-03 &6.558E-04 \\
&1.010&1.005&1.003&1.001&
\\ \hline $\vr=2^{-16}$
&1.063E-02 &5.279E-03 &2.630E-03 &1.313E-03 &6.559E-04 \\
&1.009&1.005&1.003&1.001&
 \\ \hline .&.&.&.&.&.\\.&.&.&.&.&.\\.&.&.&.&.&.
\\ \hline $\vr=2^{-30}$
&{\bf 1.063E-02} &{\bf 5.279E-03} &{\bf 2.631E-03} &{\bf 1.313E-03} &{\bf 6.559E-04} \\
&1.010&1.005&1.003&1.001&
\\ \hline $D^{N,M}$
&1.063E-02 &5.280E-03 &2.631E-03 &1.313E-03 &6.559E-04 \\
$P^{N,M}$ &1.010&1.005&1.003&1.001&\\ \hline \hline
\end{tabular}}
\end{center}
\end{table}
 Figure~\ref{fig:DiscontinuousYBC} 
displays the numerical approximation to the function $y$ defined in~\eqref{eq:ComponentY2DiscontinuousBC} and the approximation to the solution of  problem~\eqref{CproblemDiscontinuousBC} and \eqref{exDiscontinuousBC}.  Thin  boundary layer regions near $x=0$ and $x=1$ are visible in both plots, while large time derivatives near $t=0.25$ are visible only in the plot of the approximation $s+Y$.

 \begin{figure}[h!]
\centering
\resizebox{\linewidth}{!}{
	\begin{subfigure}[The computed $Y$]{
		\includegraphics[scale=0.5, angle=0]{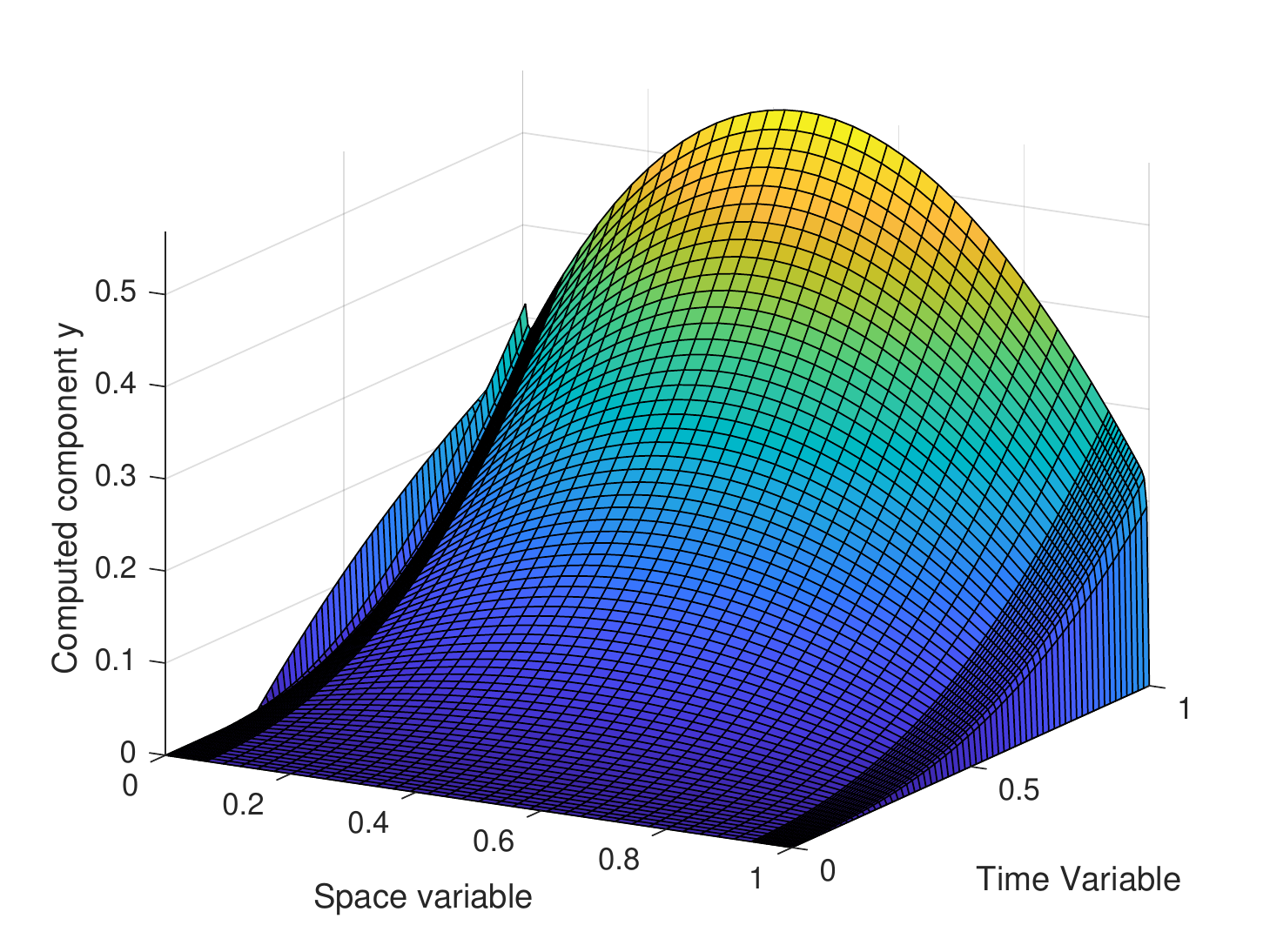}
		}
    \end{subfigure}
\begin{subfigure}[The  approximation $ s+Y$]{
		\includegraphics[scale=0.5, angle=0]{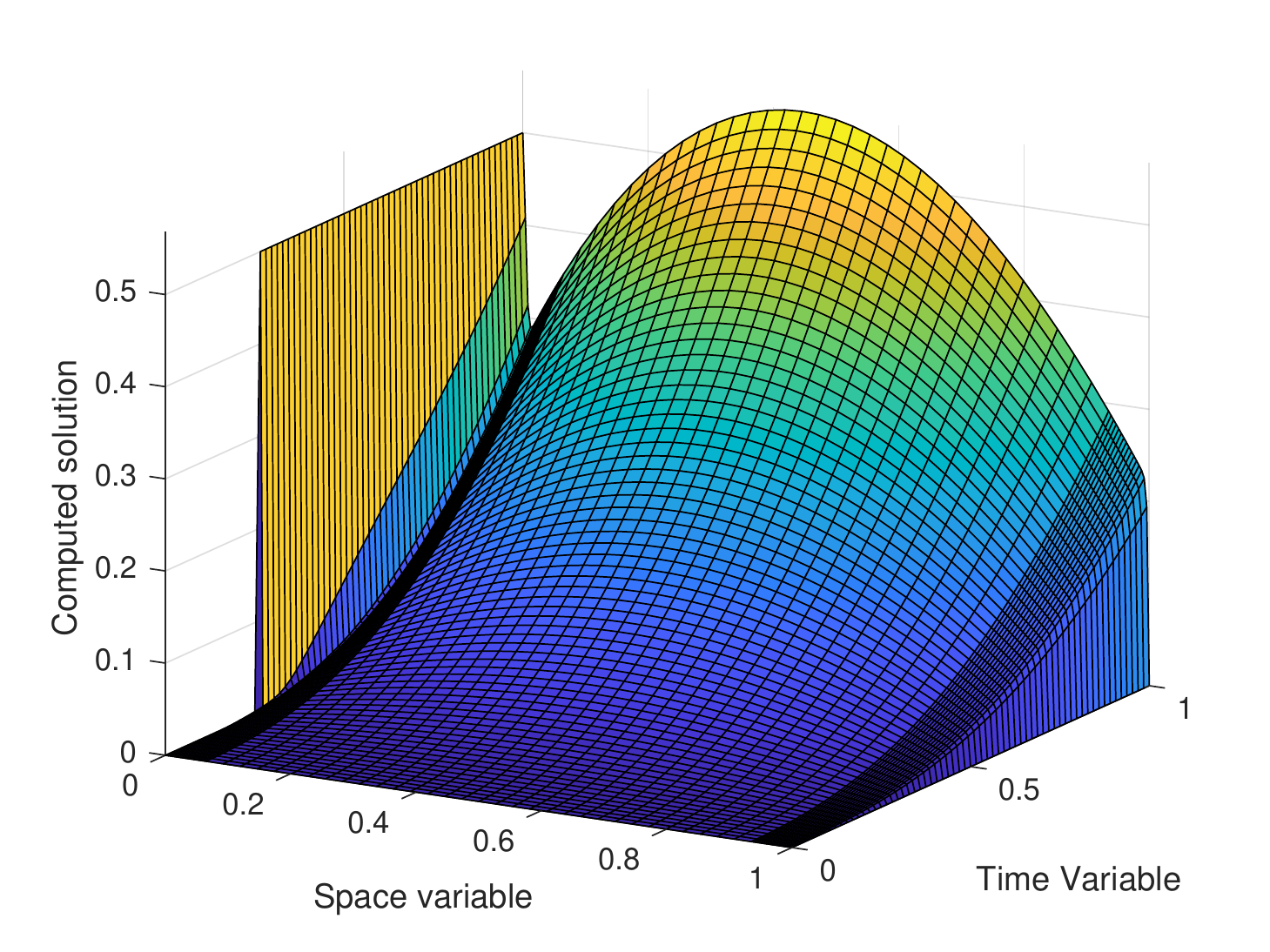}
		}
	\end{subfigure}
}
	\caption{Example~\ref{ex:exDiscontinuousBC} from Problem Class 3: The numerical approximation to $y$ and the approximation  $s+ Y$ to the solution $u$, with $\vr =2^{-16}$ and $N=M=64$}
	\label{fig:DiscontinuousYBC}
 \end{figure}

\section*{Appendix 1: Compatibility conditions}

Below we place certain regularity and compatibility restrictions on the data of the problem
\begin{subequations} \label{Classical}
\begin{align}
&  Lu:=u_t-\ve u_{xx}+b(x,t)u=f(x,t), \quad (x,t) \in  Q,\\
&u(0,t)=g_L(t), \ u(1,t)=g_R(t) \ t\ge 0, \quad u(x,0)=\phi(x), \ 0<x<1,
\end{align}
\end{subequations}
 in order that the solution   $u \in C^{4+\gamma} (\overline Q)$.
Compatibility conditions at the zero-order level  correspond to:
\begin{subequations}\label{compatibility}
\begin{equation}\label{zer0}
\phi(0^+) =g_L(0)  \quad \hbox{and} \quad \phi(1^-) =g_R(0).
\end{equation}
Assuming (\ref{zer0}), we can write $u= \Phi (x,t) + z, \ (x,t) \in  Q$ where
\begin{eqnarray*}
 \Phi (x,t) :=  \phi (x)  + (1-x)(g_L(t)-g_L (0)) + x
\left (g_R(t)-g_R (0) \right);
\end{eqnarray*}
 $ Lz = f - L \Phi;$ and $ z(x,t)=0, \ (x,t) \in \partial Q$. Note that
\begin{eqnarray*}
 L\Phi =  (1-x)g_L'(t) +x g'_R(t)  -\ve\phi''(x) +b \Phi .
\end{eqnarray*}
From \cite{ladyz}, if  $b,f, L\Phi \in C^{0+\gamma}(\bar Q)$ and the  first-order compatibility conditions
\begin{eqnarray}
  (g_R'(0)-\ve\phi''(1^-))+b(1,0) \phi( 1^-)= f(1,0), \label{first-x=1}\\  (g'_L(0)-\ve\phi ''(0^+))+b(0,0)\phi (0^+) = f(0,0), \label{first-x=0}
\end{eqnarray}
are satisfied, then $u \in C^{2+\gamma}(\bar Q)$.
 If  $b,f, L\Phi \in C^{2+\gamma}(\bar Q)$ and  we further assume  second-order compatibility (so that the mixed derivative $z_{xxt}$ is well defined at $(0,0)$ and $(1,0)$), such that
\begin{eqnarray}
 (f-L\Phi)_t(0,0^+) + (f-L\Phi)_{xx}(0^+,0)=0; \label{second-at-0}\\
  (f-L\Phi)_t(1,0^+) + (f-L\Phi)_{xx}(1^-,0)=0, \label{second-at-1}
\end{eqnarray}
then the solution of (\ref{Classical})  satisfies $u \in C^{4+\gamma}(\bar Q)$.
\end{subequations}
\newpage 

\section*{Appendix 2: Properties of $s_2(x,t)$}
 Recall that
$$
s_2(x,t):=t^2e^{-b(0) t}\erf\left(\frac{x}{2\sqrt{\vr t}} \right).
$$
\begin{subequations}\label{bounds-on-S}
Using the inequality
$
t^pe^{-t} \leq C_{p,\mu} e^{-\mu t }, \ t \in [0,\infty),\  \mu < 1, p >0;
$
we have the following bounds, for all $(x,t) \in \bar Q$,
\begin{eqnarray}
 \left \Vert  s_2 \right\Vert &\leq& C;\\
 \left \vert  \frac{\partial s_2 }{\partial t}  (x,t) \right\vert+ \ve \left \vert   \frac{\partial ^2 s_2}{\partial x^2}  (x,t) \right\vert &\leq& C\frac{xt}{\sqrt{\ve t}} e^{-\frac{x^2}{4\ve t}}\leq Ct e^{-\mu\frac{x}{\sqrt{4\ve T}}}; \\
\vr^2 \left \vert \frac{\partial ^4 s_2}{\partial x^4}  (x,t) \right \vert
&=& \vr^2  \frac{e^{-b(0)t}}{2}
 \frac{\sqrt{t}}{\ve \sqrt{\ve \pi}} \left \vert \frac{3x}{2\ve t}-\frac{x^3}{4(\ve t)^2} \right \vert e^{-\frac{x^2}{4\ve t}} \nonumber\\
&\le & Ce^{-\mu \frac{x}{\sqrt{4\ve T}}},
\end{eqnarray}
 and
\begin{align*}
\frac{\partial ^2 s_2 }{\partial t^2}  (x,t)
&= e^{-b(0)t}\frac{1}{4\sqrt{\pi}} \left( \frac{3x}{\sqrt{\ve t}} -  \frac{x^3}{2\ve t\sqrt{\ve t}}\right) e^{-\frac{x^2}{4\ve t}} \\
& +  \hbox{erf}\left(\frac{x}{2\sqrt{\ve t}}\right) 2 \left( 1-2b(0)t+\frac{(b(0)t)^2}{2} \right) e^{-b(0)t}.
\end{align*}
 Hence,
\begin{equation}
 \left\vert  \frac{\partial ^2 s_2 }{\partial t^2}  (x,t) \right\vert \leq C,\quad  x >0.
\end{equation}
The second order time derivative is bounded, but  not continuous, on the closed domain.
\end{subequations}

\end{document}